\documentclass[twoside,11pt,reqno]{amsart}
\usepackage{amsmath,amssymb,amscd}

\makeatletter

\@settopoint\oddsidemargin
\@settopoint\marginparwidth
\@settopoint\evensidemargin
\@settopoint\topmargin
\newtheorem{Proposition}{Proposition}[section]
\newtheorem{Lemma}[Proposition]{Lemma}
\newtheorem{Theorem}[Proposition]{Theorem}
\newtheorem{Corollary}[Proposition]{Corollary}

\@addtoreset{equation}{section}

\def\C{{\mathbb C}}

\def\Z{{\mathbb Z}}

\def\eps{\epsilon}
\def\sgn{{\operatorname{sgn}}}

\def\rdet{\operatorname{rdet}}
\def\row{\operatorname{row}}
\def\col{\operatorname{col}}
\def\gr{\operatorname{gr}}
\def\ev{\operatorname{ev}}
\def\pr{\operatorname{pr}}
\def\trace{\operatorname{tr}}

\newdimen\hoogte    \hoogte=14pt    
\newdimen\breedte   \breedte=15pt   
\newdimen\dikte     \dikte=0.5pt    

\newenvironment{young}{\begingroup
       \def\vr{\vrule height0.8\hoogte width\dikte depth 0.2\hoogte}
       \def\fbox##1{\vbox{\offinterlineskip
                    \hrule height\dikte
                    \hbox to \breedte{\vr\hfill##1\hfill\vr}
                    \hrule height\dikte}}
       \vbox\bgroup \offinterlineskip \tabskip=-\dikte \lineskip=-\dikte
            \halign\bgroup &\fbox{##\unskip}\unskip  \crcr }
       {\egroup\egroup\endgroup}
\def\diagram#1{\relax\ifmmode\vcenter{\,\begin{young}#1\end{young}\,}\else%
              $\vcenter{\,\begin{young}#1\end{young}\,}$\fi}

\def\mf {\mathfrak}

\begin{document}

\title[Twisted Yangians and finite $W$-algebras]{\boldmath Twisted Yangians and finite $W$-algebras}
\author{Jonathan Brown}
\address{Department of Mathematics, University of Oregon, Eugene, OR 97403.}
\email{jbrown8@uoregon.edu}

\begin{abstract}
  We construct an explicit set of generators for the finite $W$-algebras 
associated to nilpotent matrices in the
symplectic or orthogonal Lie algebras 
whose Jordan blocks are all of 
the same size.  We use these generators to show that such finite $W$-algebras
are quotients of twisted Yangians.
\end{abstract}

\maketitle

\section{Introduction and notation}
There has been renewed interest recently in the
study of finite $W$-algebras associated to nilpotent orbits
in semisimple Lie algebras; see e.g. \cite{P, Pjoseph, GG, DK, BGK, Lo}.
The goal of this paper is to show that the 
finite $W$-algebras associated to nilpotent
matrices in the symplectic or orthogonal Lie algebras
whose Jordan blocks are all of the same size
are 
homomorphic images of Olshanski's twisted Yangians from \cite{O, MNO}.
Results along these lines were first obtained by Ragoucy \cite{R} 
by a different approach, although Ragoucy was primarily concerned with the classical case, i.e. the commutative Poisson algebras that arise from the algebras
considered here on passing to their associated graded algebras. 
One new discovery in the present paper is the 
following crossover phenomenon: when the Jordan blocks are of even
size,
the finite $W$-algebra
arising from an orthogonal Lie algebra is 
a quotient of the twisted Yangian
associated to a symplectic Lie algebra and vice versa.
In \cite{BKshifted}, Brundan and Kleshchev proved an analogous result 
relating the finite $W$-algebras associated to arbitrary nilpotent elements in 
type A to quotients of so-called 
{\em shifted Yangians}. This paper is an attempt to adapt some of their 
methods to types B, C and D, specifically, the techniques from \cite[$\S$12]{BKshifted} 
dealing with nilpotent matrices whose Jordan blocks have the same size.

We begin by fixing explicit matrix realizations for the
classical Lie algebras. 
For any integer $n \geq 1$, we will label
 the rows and columns of 
$n \times n$ matrices by the ordered index set
\[
\mathcal{I}_n = \{-n+1, -n+3, \dots, n-1\}.
\]
Let $\mathfrak{g}_n = \mathfrak{gl}_n(\C)$ 
with standard basis given by the matrix units $\{e_{i,j} \:|\: i,j \in \mathcal{I}_n\}$.
Let $J_n^+$ 
be the $n \times n$ matrix with $(i,j)$ entry equal to $\delta_{i,-j}$, and
set
$$
\mf{g}_n^{+} = 
\mf{so}_n(\C)=
\{x\in \mf{g}_n\:|\:
x^T J_n^{+} + J_n^{+} x = 0\},
$$
where $x^T$ denotes the usual transpose of an $n\times n$ matrix.
Assuming in addition that $n$ is even, let
$J_n^-$ be the $n \times n$ matrix with $(i,j)$ entry equal to
$\delta_{i,-j}$ if $j > 0$ and $-\delta_{i,-j}$ if $j < 0$, and set
$$
\mf{g}_n^{-} = 
\mf{sp}_n(\C)
=\{x \in \mf{g}_n\:|\:
x^T J_n^{-} + J_n^{-} x = 0\}.
$$
We adopt the following conventions regarding signs.
For $i \in \mathcal I_n$, define
$\hat \imath \in \Z /2$ by
\begin{equation}
\hat \imath = 
\begin{cases}
0&\text{if $i \geq 0$;}\\
1&\text{if $i < 0$.}
\end{cases}
\end{equation}
We will often identify a sign $\eps =\pm$ 
with the integer
$\pm 1$ when writing formulae.
For example, $\eps^{\hat \imath}$ denotes $1$ if $\eps = +$ or $\hat \imath = 0$,
and it denotes $-1$ if $\eps = -$ and $\hat \imath = 1$.
With this notation, $\mf{g}_n^\eps$ is spanned by the matrices
$\{e_{i,j} - \eps^{\hat \imath + \hat \jmath} e_{-j,-i}\:|\:i,j \in \mathcal{I}_n\}$.

\vspace{3mm}

{\em For the remainder of the article}, we fix integers
$n, l \geq 1$ and signs $\epsilon, \phi \in \{\pm\}$, assuming
that 
$\phi = \epsilon$ if $l$ is odd,
$\phi =- \epsilon$ if $l$ is even, and
$\phi = +$ if $n$ is odd.
We will show that the finite $W$-algebra
$W^\epsilon_{n,l}$ constructed from a nilpotent matrix
of Jordan type $(l^n)$ in the Lie algebra $\mf{g}_{nl}^\epsilon$
is the level $l$ quotient of the twisted Yangian $Y_n^\phi$ 
associated to the Lie algebra $\mf{g}_n^\phi$.

\vspace{2mm}

First consider
the finite $W$-algebra side.
Let $\mathfrak{g} = \mathfrak{g}_{nl}^\epsilon$ 
and $f_{a,b} = e_{a,b} - \eps^{\hat a + \hat b} e_{-b,-a}$,
so $\mathfrak{g}$ is 
spanned by the matrices
$\{f_{a,b}\:|\:a,b \in \mathcal I_{nl}\}$.
Up to isomorphism, the finite $W$-algebra to be defined shortly
only depends on $\mathfrak{g}$ and the Jordan type $(l^n)$.
However we need to fix an explicit choice of coordinates so that we can
be absolutely explicit about the isomorphism in the main theorem below.
We do this by
introducing an $n \times l$ rectangular array of boxes, labeling
rows in order from top to bottom by the index set $\mathcal I_n$ and 
columns in order from left to right by the index set $\mathcal I_l$.
Also label the individual boxes in the array with the elements of the set
$\mathcal I_{nl}$. For $a \in \mathcal I_{nl}$ we let
$\row(a)$ and $\col(a)$ denote the row and column numbers of the box in which 
$a$ appears. We require that the boxes are labeled skew-symmetrically
in the sense that $\row(-a) = -\row(a)$ and $\col(-a) = -\col(a)$.
If $\epsilon = -$
we require in addition that $a > 0$ either if
$\col(a) > 0$ or if $\col(a) = 0$ and $\row(a) > 0$;
this additional restriction streamlines
some of the signs appearing in formulae below, notably (\ref{ev0}).
For example, if $n= 3, l = 2$ and $\epsilon = -, \phi = +$, 
one could pick the labeling
\[
  \diagram{
     {-5\phantom{:}} & {1} \cr 
 {-3\phantom{:}}     & {3} \cr
 {-1\phantom{:}}     & {5} \cr
   }
\]
and get that $\row(1) = -2$ and $\col(1) = 1$.
We remark that the above arrays are a special case of the {\em pyramids}
introduced by Elashvili and Kac in \cite{EK}; see also \cite{BG}.

Having made these choices, we let $e \in \mathfrak{g}$ denote the
following nilpotent matrix of Jordan type $(l^n)$:
\begin{equation*}
  e = 
     \sum_{
        \substack{a,b \in \mathcal{I}_{nl} \\ \row(a) = \row(b) \\ 
           \col(a) +2= \col(b) \geq 2} }
       f_{a,b}
+
\sum_{
     \substack{a,b \in \mathcal{I}_{nl} \\ 
       \row(a) = \row(b) > 0\\ \col(a)+2=\col(b)=1}}
       f_{a,b}
+
\sum_{
     \substack{a,b \in \mathcal{I}_{nl} \\ 
            \row(a) = \row(b) = 0 \\ 
            \col(a)+2=\col(b)=1 } }
       \textstyle{\frac{1}{2}}f_{a,b}.
\end{equation*}
In the above example,
$e=f_{-1,5} + \frac{1}{2} f_{-3,3} = e_{-1,5} + e_{-5,1} + e_{-3,3}$.
Also define an even grading 
\begin{equation} \label{grading}
\mf{g} = \bigoplus_{r \in \Z} \mf{g} (r)
\end{equation}
with $e \in \mf{g}(2)$
by declaring that $\deg(f_{a,b}) = \col(b) - \col(a)$.
Note this grading coincides with the grading obtained
by embedding $e$ into an $\mf{sl}_2$-triple $(e,h,f)$ and considering the
$\operatorname{ad} h$-eigenspace decomposition of $\mf{g}$.  
Let $\mf{p} = \bigoplus_{r \ge 0} \mf{g}(r)$
and $\mf{m} = \bigoplus_{r < 0} \mf{g}(r)$.
Define $\chi : \mf{m} \to \C$ by $x \mapsto \frac{1}{2} \trace(e x)$.
An explicit calculation using the formula for the nilpotent matrix $e$ 
recorded above
shows that
\begin{equation} \label{prchi}
\chi(f_{a,b}) = -\eps^{\hat a + \hat b} \chi(f_{-b,-a}) = 1
\end{equation}
if $\row(a) = \row(b),
\col(a) = \col(b)+2$ and either $\col(a) \geq 2$ or
$\col(a) = 1$, $\row(a) \geq 0$; all other $f_{a,b} \in \mf{m}$ satisfy 
$\chi(f_{a,b}) = 0$.
Let $I$ be the left ideal of the universal enveloping algebra
$U(\mf{g})$ generated by the elements $\{x - \chi(x)\:|\:x \in \mf{m}\}$.
By the PBW theorem, we have that 
$$
U(\mf{g}) = U(\mf{p}) \oplus I.
$$
Define $\pr : 
U(\mf{g}) \to U(\mf{p})$ to be the 
projection along this direct sum decomposition.
Finally the finite $W$-algebra associated to $e$ is the subalgebra
\[
  W_{n,l}^\eps 
= \{u \in U(\mf{p}) | \pr([x,u]) = 0\text{ for all }x \in \mathfrak{m} \}.
\] 
We refer the reader to the introduction of \cite{BKshifted},
where the relationship between this definition (which is 
essentially the setup of \cite{Ly}) and the more general
setup of \cite{P, GG} is explained in detail.

To make the connection between $W_{n,l}^\eps$ and the twisted Yangians, we 
exploit a shifted version of the Miura
transform, which we define as follows.
Let $\mf{h} = \mf{g}(0)$ be the Levi factor 
of $\mf{p}$ coming from
the grading. 
It is helpful
to bear in mind that
there is an isomorphism
\begin{equation}\label{h}
\qquad\:\:\mathfrak{h} \cong \begin{cases}
\mathfrak{g}_n^{\oplus m}&\text{if $l = 2m$;}\\
\mathfrak{g}_n^\eps \oplus \mathfrak{g}_n^{\oplus m}&\text{if $l = 2m+1$.}
\end{cases}
\end{equation}
Although we never need this explicitly, we note for completeness that
this isomorphism maps $f_{a,b} \in \mathfrak{h}$
to $f_{\row(a), \row(b)}\in \mathfrak{g}_n^\eps$ if $\col(a)=\col(b)=0$ or to
$e_{\row(a),\row(b)}$ in the $\lceil \frac{\col(a)}{2}\rceil$th copy of
$\mathfrak{g}_n$ if $\col(a)=\col(b)>0$.
For $q \in \mathcal{I}_l$, let
\begin{equation} \label{rho}
  \rho_q = 
   \begin{cases}  
      (n q  - \epsilon)/2 & \text{if $ q > 0$;} \\
      (n q + \epsilon)/2 & \text{if $ q < 0$;} \\
      0 & \text{if $q=0$}.
   \end{cases}
\end{equation}
Let $\eta$ be the automorphism of $U(\mf{h})$ defined on generators by
$\eta(f_{a,b}) = f_{a,b} - \delta_{a,b} \rho_{\col(a)}$.
Let $\xi : U(\mf{p}) \twoheadrightarrow U(\mf{h})$ be the algebra homomorphism induced by the
natural projection $\mf{p} \twoheadrightarrow \mf{h}$.  
The {\em Miura transform} $\mu: U(\mf{p}) \rightarrow U(\mf{h})$
is the composite map
\begin{equation} \label{miura}
  \mu = \eta \circ \xi. 
\end{equation}
By \cite[$\S$2.3]{Ly} (or Theorem \ref{miurainj} below) the restriction of $\mu$ 
to $W_{n,l}^\eps$ is injective.

\vspace{2mm}

Now we turn our attention to the twisted Yangian $Y_n^\phi$,
recalling that $\phi = -\eps$ if $l$ is even and $\phi = \eps$ if $l$ is odd.
By definition, $Y_n^\phi$ is a subalgebra of the Yangian $Y_n$.
The latter is a certain Hopf algebra over $\C$ with countably many generators 
$\{T_{i,j}^{(r)}\ | i,j \in \mathcal{I}_n, r \in \Z_{>0}\}$;
see e.g. \cite[$\S$1]{MNO} for the precise relations.  
Letting
\[
  T_{i,j}(u) = \sum_{r \ge 0} T_{i,j}^{(r)}u^{-r} \in Y_n[[u^{-1}]]
\]
where $T_{i,j}^{(0)} = \delta_{i,j}$, the comultiplication 
$\Delta:Y_n \rightarrow Y_n \otimes Y_n$ is defined by the
formula
\begin{equation}\label{com0}
  \Delta(T_{i,j}(u)) = \sum_{k \in \mathcal{I}_n}
     T_{i,k}(u) \otimes T_{k,j}(u).
\end{equation}
This and subsequent formulae involving generating functions
should be interpreted by
equating coefficients of the indeterminate $u$ on both sides of equations,
as discussed in detail in \cite[$\S$1]{MNO}.
By  \cite[$\S$3.4]{MNO}, there exists an 
automorphism $\tau:Y_n \rightarrow Y_n$
of order $2$
defined by 
$$
\tau(T_{i,j}(u)) = \phi^{\hat \imath + \hat \jmath} T_{-j,-i}(-u).
$$
We define the twisted Yangian $Y_n^\phi$ to be the subalgebra of $Y_n$ 
generated by the elements $\{S_{i,j}^{(r)} \:|\: i,j \in \mathcal{I}_n, r \in \Z_{>0}\}$
coming from the expansion
\begin{equation}\label{siju}
  S_{i,j} (u) = \sum_{r \geq 0} S_{i,j}^{(r)} u^{-r}
= 
\sum_{k \in \mathcal{I}_n} 
    \tau(T_{i,k}(u)) T_{k,j}(u) \in Y_n[[u^{-1}]].
\end{equation}
This is not the same embedding of $Y_n^\phi$ into $Y_n$ as used in
\cite[$\S$3]{MNO}: we have twisted the embedding there 
by the automorphism $\tau$.
Because of this and the fact that $\tau$ is a coalgebra antiautomorphism
of $Y_n$, we get from \cite[$\S$4.17]{MNO} that the restriction
of $\Delta$ to $Y_n^\phi$ has image contained in $Y_n^\phi \otimes Y_n$
and
\begin{equation}\label{com}
\Delta(S_{i,j}(u)) = \sum_{h,k \in \mathcal{I}_n} S_{h,k}(u) 
\otimes \tau(T_{i,h}(u)) T_{k,j}(u).
\end{equation}
We let 
$\Delta^{(m)}:Y_n \rightarrow Y_n^{\otimes (m+1)}$ 
denote the $m$th iterated comultiplication.
The preceding formula shows that it maps
$Y_n^\phi$ into $Y_n^{\phi} \otimes Y_n^{\otimes m}$.

By \cite[$\S$1.16]{MNO} there is an {evaluation homomorphism}
$Y_n \rightarrow U(\mathfrak{g}_n)$.
In view of this and (\ref{h}), we 
obtain for every $0 < p \in \mathcal{I}_l$ 
a homomorphism
\begin{equation}\label{evk}
\ev_p:Y_n \rightarrow U(\mathfrak{h}),
\qquad
T_{i,j}(u) \mapsto \delta_{i,j} +u^{-1} f_{a,b},
\end{equation}
where $a,b \in \mathcal{I}_{nl}$ are defined
from $\row(a) = i, \row(b) = j$ and $\col(a)=\col(b) = p$.
The image of this map 
is contained in the subalgebra of $U(\mathfrak{h})$
generated by the $\lceil p/2 \rceil$th copy of $\mathfrak{g}_n$ from the
decomposition (\ref{h}).
There is also an evaluation homomorphism
$Y_n^\phi \rightarrow U(\mathfrak{g}_n^\phi)$
defined in \cite[$\S$3.11]{MNO}.
If we assume that $l$ is odd (so $\eps = \phi$),
we can therefore define another homomorphism
\begin{equation}\label{ev0}
\ev_0:Y_n^\phi \rightarrow U(\mathfrak{h}),
\qquad
S_{i,j}(u) \mapsto \delta_{i,j} + (u+{\textstyle\frac{\phi}{2}})^{-1} f_{a,b},
\end{equation}
where 
$\row(a) = i, \row(b) = j$ and $\col(a) = \col(b) = 0$;
if $\eps = -$ this depends on our convention for labeling boxes as specified 
above.
The image of this map is contained in the 
subalgebra of $U(\mathfrak{h})$
generated by the subalgebra
$\mathfrak{g}_n^\eps$ in the decomposition (\ref{h}).
Putting all these things together, we deduce that there is 
a homomorphism
$$
\kappa_l:Y_n^\phi \rightarrow U(\mathfrak{h})
$$
defined by
\begin{equation} \label{kappa}
\kappa_l = 
\begin{cases}
\ev_1 \bar\otimes \ev_3 \bar\otimes\cdots\bar\otimes \ev_{l-1}
\circ \Delta^{(m)}&\text{if $l=2m+2$;}\\
\ev_0 \bar\otimes \ev_2 \bar\otimes\cdots\bar\otimes \ev_{l-1}
\circ \Delta^{(m)}&\text{if $l=2m+1$,}
\end{cases}
\end{equation}
where $\bar\otimes$ indicates composition with the natural multiplication 
in $U(\mathfrak{h})$.
We define the {\em twisted Yangian of level $l$} to be the
image of this map.
Now we are ready to state the main theorem of the article.

\begin{Theorem}\label{main1}
$\mu(W^\eps_{n,l})
=
\kappa_l(Y_n^{\phi}).$
\end{Theorem}

We will show moreover that the kernel of $\kappa_l$ 
is generated by the elements
\begin{equation} \label{ker2}
\begin{array}{rl}
  \left\{S_{i,j}^{(r)} \:\Big|\: i,j \in \mathcal{I}_n, r > l\right\}_{\phantom{S}}
    & \text{if $l$ is even;} \\  \\
\left\{ S_{i,j}^{(r)} +\frac{\phi}{2}S_{i,j}^{(r-1)}\: \Big| \: i,j \in \mathcal{I}_n, 
       r > l\right\}_{\phantom{S}} & \text{if $l$ is odd.}
\end{array}
\end{equation}
Since $W^\eps_{n,l} \cong \mu(W^\eps_{n,l})$ by injectivity of the
Miura transform,
and a full set of relations between the generators $S_{i,j}^{(r)}$
of $Y_n^\phi$ are known by \cite[$\S$3.8]{MNO}, 
this means that we have found a 
full set of generators and relations for the finite $W$-algebra
$W_{n,l}^\eps$.

\vspace{2mm}

The key step in our proof of Theorem~\ref{main1}
is a remarkable explicit formula for the generators of 
$W_{n,l}^\eps$ corresponding to the elements $S_{i,j}^{(r)}\in Y_n^\phi$.
In the remainder of the introduction, we want to 
explain this formula.
Given $i,j \in \mathcal{I}_n$ 
and $p,q \in \mathcal{I}_l$,
let $a,b$ be the elements of $\mathcal{I}_{n l}$ such that
$\col(a) = p$,
$\col(b) = q$, 
$\row(a) = i$, and
$\row(b) = j$.
Define a linear map
$s_{i,j} : \mf{g}_l \rightarrow \mf{g}$ by setting
\begin{equation} \label{sij}
    s_{i,j}(e_{p,q}) = \phi^{\hat \imath \hat p + \hat \jmath \hat q} 
f_{a,b}.
\end{equation}
Let $M_n$ denote the algebra of $n \times n$ matrices
over $\C$, with rows
and columns labeled by the index set $\mathcal{I}_n$ as usual,
and let $T(\mf{g}_l)$ be the tensor algebra on the vector space
$\mf{g}_l$.
Let 
\begin{equation} \label{S}
  s : T(\mf{g}_l) \to M_n \otimes U(\mf{g})
\end{equation}
be the algebra homomorphism that maps a generator
$x \in \mf{g}_l$ to
$\sum_{i,j \in \mathcal{I}_n} e_{i,j} \otimes s_{i,j}(x)$.
This in turn defines linear maps
\begin{equation*}
  s_{i,j} :  T(\mf{g}_l) \to U(\mf{g}),
\end{equation*}
such that
\[
s(x) = \sum_{i,j \in \mathcal{I}_n} e_{i,j} \otimes s_{i,j}(x)
\]
for every $x \in T(\mf{g}_l)$.
Note for any $x, y \in T(\mf{g}_l)$ that
\begin{equation} \label{sijmult}
  s_{i,j}(x y) = \sum_{k \in \mathcal{I}_n} s_{i,k}(x) s_{k,j}(y)
\end{equation}
and also
$s_{i,j}(1) = \delta_{i,j}$.

If $A$ is an $l \times l$ 
matrix with entries in some ring,
we define its {\em row determinant} $\rdet A$
to be the usual Laplace expansion of determinant, but keeping
the (not necessarily commuting) monomials that arise in {\em row order}; 
see e.g. \cite[(12.5)]{BKshifted}.
For $q \in \mathcal{I}_l$ and an indeterminate $u$,
let 
$$
u_q = u+ e_{q,q} +\rho_q \in T(\mf{g}_l)[u],
$$
recalling the definition of $\rho_q$ from \eqref{rho}.
Define $\Omega(u)$ to be the $l \times l$ matrix with
entries in $T(\mf{gl}_l)[u]$ whose $(p,q)$ entry for $p,q \in \mathcal I_l$
is equal to
\begin{equation}\label{omega_even}
\Omega(u)_{p,q} =
\left\{
\begin{array}{ll}
e_{p,q}&\text{if $p < q$;}\\
u_q&\text{if $p=q$;}\\
-1&\text{if $p = q+2 < 0$;}\\
-\phi&\text{if $p  = q+2 = 0$;}\\
1&\text{if $p = q+2 > 0$;}\\
0&\text{if $p  > q+2$.}
\end{array}\right.
\end{equation}
For example, if $l = 4$ then
$$
\Omega(u) = \left(
\begin{array}{cccc}
u_{-3}&e_{-3,-1}&e_{-3,1}&e_{-3,3}\\
-1&u_{-1}&e_{-1,1}&e_{-1,3}\\
0&1&u_1&e_{1,3}\\
0&0&1&u_3
\end{array}
\right).
$$
If $l$ is odd we also need the $l \times l$ matrix
$\bar\Omega(u)$ defined by
\begin{equation} \label{bar_omega}
  \bar\Omega(u)_{p,q} = 
   \begin{cases}
     \Omega(u)_{p,q} & \text{if $p \ne 0$ or $q \ne 0$;} \\
     e_{0,0}   & \text{if $p = q = 0$.}
    \end{cases}
\end{equation}
For example, if $l=5$ then
\begin{align*}
\Omega(u)&=
\left(
\begin{array}{ccccc}
u_{-4}&e_{-4,-2}&e_{-4,0}&e_{-4,2}&e_{-4,4}\\
-1&u_{-2}&e_{-2,0}&e_{-2,2}&e_{-2,4}\\
0&-\phi&u_0&e_{0,2}&e_{0,4}\\
0&0&1&u_2&e_{2,4}\\
0&0&0&1&u_4
\end{array}
\right),\\
\bar \Omega(u)&=
\left(
\begin{array}{ccccc}
u_{-4}&e_{-4,-2}&e_{-4,0}&e_{-4,2}&e_{-4,4}\\
-1&u_{-2}&e_{-2,0}&e_{-2,2}&e_{-2,4}\\
0&-\phi&e_{0,0}&e_{0,2}&e_{0,4}\\
0&0&1&u_2&e_{2,4}\\
0&0&0&1&u_4
\end{array}
\right).
\end{align*}
Then we let
\begin{equation}\label{omegadef}
  \omega(u) = 
\sum_{r = -\infty}^{l} \omega_{l-r} u^{r}  = 
\left\{
\begin{array}{ll}
\rdet \Omega(u)&\text{if $l$ is even;}\\
\displaystyle\rdet \Omega(u) + \sum_{r=1}^{\infty} (-2\phi u)^{-r} 
\rdet \bar\Omega(u)  &\text{if $l$ is odd.}
\end{array}\right.
\end{equation}
This defines elements $\omega_r \in T(\mathfrak{g}_l)$,
hence elements $s_{i,j}(\omega_r) \in U(\mathfrak{g})$
for $i,j \in \mathcal{I}_n$ and $r \geq 1$.
It is obvious from the definition that
each $s_{i,j}(\omega_r)$ actually belongs to $U(\mathfrak{p})$.

\begin{Theorem} \label{gens}
The elements 
$\{s_{i,j}(\omega_r)\:|\:i, j \in \mathcal{I}_n, r \geq 1\}$
generate the subalgebra $W_{n,l}^\eps$. Moreover,
$\mu(s_{i,j}(\omega_r)) = \kappa_l(S_{i,j}^{(r)})$.
\end{Theorem}

The hardest part of the proof is to show that
each $s_{i,j}(\omega_r)$ belongs to $W_{n,l}^\eps$.
This is established by a lengthy calculation which we postpone until $\S$4.
In $\S$2 we study
the twisted Yangian of level $l$, in particular proving a PBW
theorem for this algebra and computing the kernel of $\kappa_l$
as mentioned above.
We also check that
$\mu(s_{i,j}(\omega_r))  = \kappa_l(S_{i,j}^{(r)})$. 
Then in $\S$3 we complete the proofs of Theorems~\ref{main1} and \ref{gens}.
At the same time we obtain a direct proof of the injectivity of the
Miura transform in this case.

In subsequent work,
we will combine the results of this article with work of Molev \cite{Molev}
to deduce the classification of finite dimension irreducible representations
of the finite $W$-algebras $W_{n,l}^\eps$;
we expect this will allow us to verify \cite[Conjecture 5.2]{BGK} in 
this case. It seems possible that the connection to finite $W$-algebras
could also be used to derive explicit character formulae for the finite dimensional irreducible representations of twisted Yangians, as was done in type A
in \cite[$\S$8]{BKrep}.

\vspace{2mm}

{\em Acknowledgments.}  The author would like to thank Jonathan Brundan for
suggesting this problem and for his generous advice while 
writing this article, and Alexander Molev for some helpful comments.

\section{The twisted Yangian of level $l$}

Continuing with notation from the introduction,
we begin this section by giving a different
description of the 
map $\kappa_l:Y_n^\phi \rightarrow U(\mathfrak{h})$ from (\ref{kappa}).
Let
\begin{align*}
T(u) &= \sum_{i,j \in \mathcal I_n} e_{i,j} \otimes T_{i,j}(u) \in
M_n \otimes Y_n[[u^{-1}]],\\
S(u) &= \sum_{i,j \in \mathcal I_n} e_{i,j} \otimes S_{i,j}(u) \in
M_n \otimes Y^\phi_n[[u^{-1}]].
\end{align*}
For a linear map $f:V \rightarrow W$,
we use the same notation $f$ for the induced map 
$\operatorname{id} \otimes f:
M_n \otimes V \rightarrow M_n \otimes W$.
Thinking of elements of 
$M_n \otimes V$ 
(resp. $M_n \otimes W$) 
as $n \times n$ matrices with entries in $V$ (resp. $W$),
this is just the linear map obtained by applying $f$ 
simultaneously to all matrix entries.
We extend (\ref{evk}) by defining
a homomorphism $\ev_{-p}:Y_n \rightarrow U(\mathfrak{h})$ 
for $0 < p \in \mathcal{I}_l$ by setting
\begin{equation}\label{Evm}
\ev_{-p} = \ev_p \circ \tau.
\end{equation}
Since the images of $\ev_p$ and $\ev_q$ 
commute for $p \ne \pm q$, it is then the case by (\ref{kappa}), (\ref{com0}), (\ref{siju}) and
(\ref{com}) 
that
\begin{multline} \label{kappa2}
\kappa_l(S(u)) = \\
\begin{cases}
\ev_{1-l}(T(u)) \cdots \ev_{-1}(T(u)) \ev_1(T(u)) \cdots \ev_{l-1}(T(u))
&\text{if $l$ is even;}\\
\ev_{1-l}(T(u)) \cdots \ev_{-2}(T(u)) 
\ev_0(S(u)) \ev_2(T(u)) \cdots \ev_{l-1}(T(u))
&\text{if $l$ is odd,}
\end{cases}
\end{multline}
where the product on the right hand side is 
in the algebra 
$M_n \otimes U(\mathfrak{h})[[u^{-1}]]$.

For any $0 \neq p \in \mathcal{I}_l$,
(\ref{Evm}), (\ref{evk}), and the labeling convention for boxes implies that
\begin{equation*}
\ev_p(T_{i,j}(u))
= \delta_{i,j} +u^{-1} \phi^{\hat p(\hat \imath + \hat \jmath)} 
f_{a,b},
\end{equation*}
where $a,b \in \mathcal{I}_{nl}$ satisfy
$\row(a) = i, \row(b) = j$ and $\col(a) = \col(b) = p$.
Hence in the notation (\ref{sij}) we have that
\begin{equation*}
\ev_p(T_{i,j}(u)) =  \delta_{i,j} + u^{-1} s_{i,j}(e_{p,p}).
\end{equation*}
Also (\ref{ev0}) is equivalent to
\begin{equation*}
\ev_0(S_{i,j}(u)) = \delta_{i,j} + (u + {\textstyle \frac{\phi}{2}})^{-1}
s_{i,j}(e_{0,0})
=\delta_{i,j} +\sum_{r =0}^\infty
(-2\phi )^{-r}u^{-1-r} s_{i,j}(e_{0,0}).
\end{equation*}
Using the more sophisticated notation
(\ref{S}), we deduce that
\begin{align*}
u \ev_p(T(u)) &= s(u+e_{p,p}),\\
u \ev_0 (S(u)) &= s(u+e_{0,0}) +
\sum_{r = 1}^\infty (-2\phi u)^{-r} s(e_{0,0}).
\end{align*}
Hence (\ref{kappa2}) is equivalent to the equation
\begin{equation}\label{pf1}
u^l \kappa_l(S(u))
=
s((u+e_{1-l,1-l}) \cdots (u+e_{-1,-1})
(u+e_{1,1})\cdots (u+e_{l-1,l-1}))
\end{equation}
if $l$ is even and
\begin{multline}\label{pf2}
u^l \kappa_l(S(u))
=
s((u+e_{1-l,1-l}) \cdots (u+e_{-2,-2})
(u+e_{0,0})
(u+e_{2,2})\cdots (u+e_{l-1,l-1}))\\
+
\sum_{r = 1}^\infty
(-2\phi u)^{-r} s((u+e_{1-l,1-l}) \cdots (u+e_{-2,-2})
e_{0,0}
(u+e_{2,2})\cdots (u+e_{l-1,l-1}))
\end{multline}
if $l$ is odd.
Equating $u^{l-r}$-coefficients gives that
\begin{multline}\label{kappa3}
\kappa_l(S_{i,j}^{(r)})
= 
\sum_{\substack{
p_1,\dots,p_r \in \mathcal{I}_l \\
p_1 < \cdots < p_r}}
s_{i,j}(e_{p_1,p_1}\cdots e_{p_r,p_r})
+
\sum_{t=1}^{r-1}
(-2\phi)^{t-r}
\!\!\!\sum_{\substack{
p_1,\dots,p_t \in \mathcal{I}_l \\
p_1 < \cdots < p_t \\
0 \in \{p_1,\dots, p_t\}}}
s_{i,j}(e_{p_1,p_1}\cdots e_{p_t,p_t}),
\end{multline}
the last term in this formula being
zero automatically if $l$ is even.
The following theorem verifies the second statement of Theorem~\ref{gens}.

\begin{Theorem} \label{images}
$u^l \kappa_l(S(u)) = \mu(s(\omega(u)))$.
\end{Theorem}

\begin{proof}
The Miura transform (\ref{miura}) satisfies
$\mu(s(u_{p})) = s(u+e_{p,p})$
and
$\mu(s(e_{p,q})) = 0$
if $p < q$.
So recalling 
the matrices $\Omega(u)$ and $\bar\Omega(u)$ from \eqref{omega_even}
and
\eqref{bar_omega}
we get that
\[
  \mu(s(\rdet \Omega(u))) = 
      s( (u+e_{1-l,1-l})\cdots (u+e_{l-1,l-1})),
\]
and
\[
  \mu(s(\rdet \bar\Omega(u))) = 
     s((u+e_{1-l,1-l}) \cdots (u+e_{-2,-2})
      e_{0,0} (u+e_{2,2}) \cdots (u+e_{l-1,l-1})).
\]
The theorem follows on 
comparing  (\ref{omegadef}), (\ref{pf1}) and (\ref{pf2}).
\end{proof}

The goal now is to prove a PBW theorem for the twisted Yangian 
of level $l$, $\kappa_l(Y_n^\phi)$.
We will need the following
elementary lemma, which is established in the proof of 
\cite[Theorem 3.1]{BKparabpres}.

\begin{Lemma} \label{bk3.1}
  Let $X$ be the variety of tuples
$(A_{1-l}, A_{3-l}, \dots, A_{l-1})$ of $n \times n$ matrices.
Let $x_{i,j}^{[r]} \in \C[X]$ be the coordinate function picking
out the $(i,j)$ entry of $A_r$.  Let $Y$ be the variety of tuples
$(B_1, \dots, B_l)$ of $n \times n$ matrices.
Let $y_{i,j}^{[r]} \in \C[Y]$ be the coordinate function picking out
the $(i,j)$ entry of $B_r$.
Define
\[
  \theta : X \to Y, \quad (A_{1-l}, \dots, A_{l-1}) \mapsto 
    (B_1, \dots, B_l)
\]
where
\[
  B_r = \sum_{\substack{p_1, \dots, p_r \in \mathcal{I}_l \\
      p_1 < \dots < p_r }}
     A_{p_1} A_{p_2} \dots A_{p_r},
\]
that is, $B_r$ is the $r$th elementary symmetric function in the matrices
$(A_{1-l}, \dots, A_{l-1})$.
Then the comorphism $\theta^*: \C[Y] \to \C[X]$ satisfies
\[
  \theta^*(y_{i,j}^{[r]}) =
  \sum_{\substack{
  i_1, \dots, i_{r-1} \in \mathcal{I}_n \\
  p_1, \dots, p_r \in \mathcal{I}_l \\
  p_1 < \dots < p_r }}
  x_{i,i_1}^{[p_1]} x_{i_1, i_2}^{[p_2]} \dots x_{i_{r-1},j}^{[p_r]}
\]
Moreover the derivative $d \theta _x : T_x(X) \to T_{\theta(x)}(Y)$
is an isomorphism for any point
$x = (c_{1-l} I_n, \dots, c_{l-1} I_n)$ such that
$c_{1-l}, \dots, c_{l-1}$ are pairwise distinct scalars.
\end{Lemma}

We observe by \eqref{kappa3} for $i,j \in \mathcal{I}_n$ that
\begin{equation} \label{kernel}
\begin{cases}
  \kappa_l(S_{i,j}^{(r)}) = 0 & \text{if $l$ is even and $r > l$;} \\
  \kappa_l(S_{i,j}^{(r)}) = -\frac{\phi}{2} \kappa_l(S_{i,j}^{(r-1)}) &
     \text{if $l$ is odd and $r > l$.}
\end{cases}
\end{equation}
Following \cite[$\S$3.14]{MNO}, we say
$(i,j,r)$ is {\em admissible} if $i,j \in \mathcal{I}_n$, 
$1 \le r \le l$, and
\[
  \begin{cases}
  i + j \le 0 & \text{if $\phi = +$ and $r$ is even;} \\
  i + j < 0 & \text{if $\phi = +$ and $r$ is odd;} \\
  i + j < 0 & \text{if $\phi = -$ and $r$ is even;} \\
  i + j \le 0 & \text{if $\phi = -$ and $r$ is odd.} \\
  \end{cases}
\]
Now consider the standard filtration on $U(\mf{h})$ defined by declaring 
that each $x \in \mf{h}$ is in degree $1$.  This induces a filtration
on the subalgebra $\kappa_l(Y_n^\phi)$ so that 
$\gr \kappa_l(Y_n^\phi)$ is a subalgebra of $\gr U(\mf{h})$.
Note by (\ref{kappa3}) that each $\kappa_l(S_{i,j}^{(r)})$ belongs to the
filtered degree $r$ component of $U(\mathfrak{h})$.

\begin{Theorem} \label{freely_gen}
The 
elements
$
\left\{\gr_r \kappa_l(S_{i,j}^{(r)}) 
  \:\Big|\: (i,j,r) \: \text{is admissible} \right\}
$
are algebraically independent generators for
the commutative algebra $\gr \kappa_l(Y_n^\phi)$.
Hence the monomials in the elements 
$
\left\{\kappa_l(S_{i,j}^{(r)})
  \:\Big|\: (i,j,r) \: \text{is admissible} \right\}
$
taken in some fixed order form a basis for $\kappa_l(Y_{n}^{\phi})$.
\end{Theorem}

\begin{proof}
As in \cite[(3.6.4)]{MNO},
we have for all $i, j \in \mathcal{I}_n$ the following relation in $Y_n^\phi[[u^{-1}]]$:
\begin{equation} \label{sij_rel}
  \phi^{\hat \imath + \hat \jmath}  S_{-j,-i}(-u) =
     S_{i,j}(u) + \phi \frac{S_{i,j}(u) - S_{i,j}(-u)}{2 u}.
\end{equation}
By \eqref{kernel} and \eqref{sij_rel}
monomials in the elements 
$\left\{\gr_r \kappa_l(S_{i,j}^{(r)}) \: \Big| \: (i,j,r) \text{ is admissible}\right\}$ taken
in some fixed order
generate $\gr \kappa_l(Y_n^\phi)$, so it suffices to prove 
they are algebraically independent.
Let notation be as in Lemma \ref{bk3.1}.
Let $V$ be the closed subspace of
$X$ defined by the ideal $I$ generated by
\[
  \left\{ x_{i,j}^{[r]} + \phi^{\hat \imath + \hat \jmath} x_{-j,-i}^{[-r]}
  \: \Big| \: i,j \in \mathcal{I}_n, r \in \mathcal{I}_l\right\}.
\]
As $\mf{h}$ is the vector space
spanned by 
$\{s_{i,j}(e_{p,p})\: | \: i, j \in \mathcal{I}_n, p \in \mathcal{I}_l\}$
subject only to the relations
$
s_{i,j}(e_{p,p}) = - \phi^{\hat \imath + \hat \jmath} s_{-j,-i}(e_{-p,-p}),
$
we can identify $\gr U(\mf{h})$ with $\C[V]$, by declaring that
$\gr_1 s_{i,j}(e_{p,p}) = x_{i,j}^{[p]}+I$.

Let $W$ be the closed subspace of $Y$ 
defined by the ideal $J$ generated by
$$
\left\{y_{i,j}^{[r]} - (-1)^r \phi^{\hat \imath + \hat \jmath} y_{-j,-i}^{[r]}
\:\Big|\: i,j \in \mathcal{I}_n, r=1,\dots,l\right\}.
$$
We claim that $\theta(V) \subseteq W$, i.e. $\theta^*(J) \subseteq I$.
To see this note that
\begin{align*}
\theta^* (y_{i,j}^{[r]}) &= 
  \sum_{\substack{
  i_1, \dots, i_{r-1} \in \mathcal{I}_n \\
  p_1, \dots, p_r \in \mathcal{I}_l \\
  p_1 < \dots < p_r }}
  x_{i,i_1}^{[p_1]} x_{i_1, i_2}^{[p_2]} \dots x_{i_{r-1},j}^{[p_r]} \\
  &\equiv  (-1)^r \phi^{\hat \imath + \hat \jmath} \sum_{\substack{
  i_1, \dots, i_{r-1} \in \mathcal{I}_n \\
  p_1, \dots, p_r \in \mathcal{I}_l \\
  p_1 < \dots < p_r }}
  x_{-j,-i_{r-1}}^{[-p_r]} x_{-i_{r-1}, -i_{r-2}}^{[-p_{r-1}]} 
    \dots x_{-i_{1},-i}^{[-p_1]} 
  \pmod{I} \\
  &\equiv (-1)^r \phi^{\hat \imath + \hat \jmath} \theta^*(y_{-j,-i}^{[r]}) \pmod{I}.
\end{align*}
Hence
$\theta^*\left(y_{i,j}^{[r]} - (-1)^r \phi^{\hat \imath+ \hat \jmath} 
y_{-j,-i}^{[r]}\right) \in I$.

Choose $x = (c_{1-l} I_n , \dots, c_{l-1} I_n) \in X$ so that
$c_{1-l}, \dots, c_{l-1}$ are pairwise distinct and
$c_{i} + c_{-i} = 0$.  Then $x$ belongs to $V$.
Now apply Lemma \ref{bk3.1} to deduce that
$d \theta_x: T_x (V) \to T_{\theta(x)}(W)$ is injective.
An easy calculation shows that $\dim V = \dim W$, hence
$d \theta_x: T_x (V) \to T_{\theta(x)}(W)$ is an isomorphism.
By \cite[Theorem 4.3.6(i)]{Springer} this implies that
$\theta:V \rightarrow W$ is a dominant morphism, so the comorphism
$\theta^*:\C[W] \rightarrow \C[V] = \gr U(\mathfrak{h})$ is injective.
As $\C[W]$ is freely generated by the elements
$\left\{y_{i,j}^{[r]} \: \Big| \: (i,j,r) \: \text{is admissible} \right\}$, we deduce that the elements
$\left\{\theta^*(y_{i,j}^{[r]}) \: \Big| \: (i,j,r) \: \text{is admissible} \right\}$
are algebraically independent too.  
It remains to observe by applying $\gr_r$ to (\ref{kappa3}) 
and using \eqref{sijmult}
that
$$
\gr_r \kappa_l(S_{i,j}^{(r)}) =
  \sum_{\substack{
  i_1, \dots, i_{r-1} \in \mathcal{I}_n \\
  p_1, \dots, p_r \in \mathcal{I}_l \\
  p_1 < \dots < p_r }}
  x_{i,i_1}^{[p_1]} x_{i_1, i_2}^{[p_2]} \dots x_{i_{r-1},j}^{[p_r]}
=\theta^*(y_{i,j}^{[r]}).
$$
\end{proof}

\begin{Corollary}
The elements
\begin{equation}\label{copy}
\begin{array}{rl}
  \left\{S_{i,j}^{(r)} \:\Big|\: i,j \in \mathcal{I}_n, r > l\right\}_{\phantom{S}}
    & \text{if $l$ is even;} \\  \\
\left\{ S_{i,j}^{(r)} +\frac{\phi}{2}S_{i,j}^{(r-1)}\: \Big| \: i,j \in \mathcal{I}_n, 
       r > l\right\}_{\phantom{S}} & \text{if $l$ is odd}
\end{array}
\end{equation}
generate the kernel of $\kappa_l$.
\end{Corollary}
\begin{proof}
Let $I$ denote the two-sided ideal of $Y_n^{\phi}$ generated by
the elements listed in (\ref{copy}).
It is obvious that $\kappa_l$ induces a map 
$\bar \kappa_l : Y_n^{\phi} / I \to \kappa_l(Y_n^{\phi})$.  
Since $Y_n^{\phi} / I$ is spanned by the set of all monomials in the elements
$\left\{S_{i,j}^{(r)} + I \:\Big|\: (i,j,r)\text{ is admissible}\right\}$
taken in some fixed order by \cite[$\S$3.14]{MNO}, 
and the images of these monomials are linearly independent in
$\kappa_l(Y_n^\phi)$ by Theorem~\ref{freely_gen},
we deduce that $\bar \kappa_l$ is an isomorphism.
\end{proof}

We also obtain a new proof of the PBW theorem for twisted Yangians, different from the
one in \cite[$\S$3]{MNO}.
\begin{Corollary}
The set of all monomials in the elements $\left\{S_{i,j}^{(r)} \: \Big| \: (i,j,r) \text{ is admissible} \right\}$
taken in some fixed order forms a basis for $Y_n^{\phi}$.
\end{Corollary}
\begin{proof}
It is clear from \eqref{sij_rel} that such monomials span $Y_n^\phi$.  The fact that they are linearly 
independent follows from
Theorem \ref{freely_gen} by taking sufficiently large $l$.
\end{proof}

\section{The finite $W$-algebra}

In $\S4$ below we will prove the following theorem:

\begin{Theorem} \label{containedinW}
For $ i,j \in \mathcal{I}_n$ and $r \geq 1$,
the element
$s_{i,j}(\omega_r)$ belongs to $W_{n,l}^{\epsilon}$.
\end{Theorem}

In the remainder of this section we explain 
how to deduce the main results formulated in the
introduction from this theorem. 

The finite $W$-algebra $W_{n,l}^\eps$ 
possesses two natural filtrations. 
The first of these, the {\em Kazhdan filtration},
is the filtration on $W_{n,l}^\eps$ induced by the filtration
on $U(\mathfrak{g})$ generated by declaring that
each element $x \in \mf{g}(r)$ in the grading (\ref{grading})
is of 
degree $r/2+1$. 
The fundamental {\em PBW theorem} for finite $W$-algebras 
asserts that the associated
graded algebra $\gr W_{n,l}^\eps$
under the Kazhdan filtration 
is isomorphic to the coordinate
algebra of the Slodowy slice  at 
$e$; see e.g. \cite[Theorem 4.1]{GG}.

The second important filtration, called the {\em good filtration}
in [BGK], is 
defined as follows.
The grading (\ref{grading})
induces a non-negative grading on $U(\mathfrak{p})$.
Although $W_{n,l}^\eps$ is not a graded subalgebra of
$U(\mathfrak{p})$, this grading on $U(\mathfrak{p})$ still induces
a filtration on $W_{n,l}^\eps$ 
with respect to which the associated graded algebra
$\gr' W_{n,l}^\eps$ 
is naturally identified with a 
graded subalgebra of $U(\mf{p})$.
The fundamental result about the good filtration, which is a
consequence of the PBW theorem and \cite[(2.1.2)]{Pjoseph}, is that
\begin{equation} \label{assgr}
\gr' W_{n l}^\eps = U(\mf{g}_e)
\end{equation}
as subalgebras of $U(\mathfrak{p})$,
where $\mf{g}_e$ denotes the centralizer of $e$ in $\mf{g}$;
see also \cite[Theorem 3.5]{BGK}.
The element $s_{i,j}(\omega_{r+1})$ belongs to the subspace of elements of degree $r$
in the good filtration, and we have
that $s_{i,j}(\omega_{r+1}) \in W_{n,l}^\eps$
by Theorem~\ref{containedinW}.
So it makes sense to define
\begin{equation}\label{cij}
f_{i,j; r} = \gr'_r s_{i,j}(\omega_{r+1}) \in U(\mathfrak{g}_e)
\end{equation}
for $r \geq 0$.
Explicitly, we have that
\begin{equation} \label{fijr}
f_{i,j;r} = 
 \sum_{\substack{p, q \in \mathcal{I}_l \\q-p=2r}}
  \alpha_{p,q} s_{i,j}(e_{p, q})
\end{equation}
where
\[
  \alpha_{p,q} = 
  \begin{cases}
    1 & \text{if $q < 0$;} \\
    \phi (-1)^{q/2} & \text{if $p < 0$ and $q \ge 0$ and $l$ is odd;} \\
    (-1)^{(q+1)/2} & \text{if $p < 0$ and $q > 0$ and $l$ is even;} \\
    (-1)^{(q-p)/2} & \text{if $p \ge 0$.}
  \end{cases}
\]
This formula comes from the fact that the monomial $e_{p,q}$ where $q-p=2r$
occurs 
in $\rdet \Omega(u)$ as a coefficient of $u^{l-(r+1)}$ (and thus in
$\omega_{r+1}$) because of the element 
$\sigma = (p, q, q-2, \dots, p+2)$ in
the symmetric group on $\mathcal{I}_l$.
Now $\alpha_{p,q} = \sgn(\sigma) * N$, where $N$ is the number of 
$-1$'s strictly below and strictly to the left of $e_{p,q}$ in the matrix $\Omega(u)$.

So \eqref{fijr} shows that each $f_{i,j;r} \in U(\mathfrak{g}_e)$ is 
an element of $\mathfrak{g}$,
hence belongs to $\mathfrak{g}_e$.

\begin{Lemma} \label{centralizer}
The elements $\{f_{i,j;r} \: | \: (i,j, r+1) \: \text{is admissible}\}$
form a basis for $\mf{g}_e$.
\end{Lemma}
\begin{proof}
We have already observed that each $f_{i,j;r}$ belongs to
$\mathfrak{g}_e$.
By \cite[$\S$3.2]{J}, the dimension of $\mf{g}_e$ is
\[
  \begin{cases}
   n^2 l / 2 & \text{if $l$ is even}; \\
   (n^2 l - n \eps)/2 & \text{if $l$ is odd}.
  \end{cases}
\]
An easy calculation shows that this is the same as the number
of admissible triples.
Now it just remains to show that the elements
$f_{i,j;r}$ for all admissible $(i,j,r+1)$
are linearly independent.
This is easy to see on noting that all these elements are non-zero,
which follows by computing some explicit matrix coefficients.
\end{proof}

\begin{Theorem} \label{gen_cor}
The elements 
$\{s_{i,j}(\omega_r) \: | \: (i,j,r) \: \text{is admissible} \}$
generate $W_{n,l}^\eps$.
\end{Theorem}

\begin{proof}
By \eqref{assgr}, \eqref{cij} and Lemma \ref{centralizer}, the elements
$$
\{\gr'_r s_{i,j}(\omega_{r+1}) \: | \: (i,j,r+1) \: \text{is admissible} \}
$$
generate $\gr' W_{n,l}^\eps$,
the associated graded algebra in the good filtration.
The theorem follows from this statement by induction on the filtration.
\end{proof}

Theorems~\ref{main1} and \ref{gens} from the introduction
follow from Theorems \ref{containedinW}, \ref{gen_cor}
 and \ref{images}.
Finally we include a proof of the following theorem,
which is originally due to \cite[Corollary 2.3.2]{Ly}
in a more general setting.

\begin{Theorem} \label{miurainj}
The Miura transform $\mu : W_{n,l}^{\eps} \to U(\mf{h})$ 
from (\ref{miura}) is injective.
\end{Theorem}
\begin{proof}
Note that $\mu$ is a filtered map with respect to the Kazhdan filtration
on $W_{n,l}^{\eps}$ and the
standard filtration on
$U(\mf{h})$. 
We actually show that the associated graded map
$\gr \mu: \gr W_{n,l}^\eps \rightarrow \gr U(\mf{h})$
is injective, which implies the theorem.
Each $s_{i,j}(\omega_r)$ is in
degree $r$ under the Kazhdan filtration and  
$\kappa_l(S_{i,j}^{(r)})$ is in degree $r$ under the standard
filtration on $U(\mf{h})$.
Moreover Theorem \ref{images} 
shows that $\mu(s_{i,j}(\omega_r)) = \kappa_l(S_{i,j}^{(r)})$,
hence
$(\gr \mu)(\gr_r s_{i,j}(\omega_r)) =  \gr_r \kappa_l(S_{i,j}^{(r)})$.
So by Theorem \ref{freely_gen} and the PBW theorem for
$W_{n,l}^\eps$
we deduce that $\gr \mu: \gr W_{n,l}^\eps \rightarrow \gr U(\mf{h})$ is injective. 
\end{proof}

\section{Proof of invariance}

In this section we prove Theorem \ref{containedinW}.
We need to show for $i,j \in \mathcal{I}_n$ and $r \geq 1$
that
\begin{equation}\label{toprove}
\pr([x, s_{i,j}(\omega_r)]) = 0
\end{equation}
for all $x \in \mathfrak{m}$.
Since $\mf{m}$ is generated by the elements
\begin{equation}\label{thegens}
\left\{s_{i,j}(e_{q+2,q}) \:|\: 
   i,j \in \mathcal{I}_n,
  q \in \mathcal{I}_l, -1 \le q < l-1\right\},
\end{equation}
we just 
need to consider the actions of these elements on 
each $s_{i,j}(\omega_r)$.
Actually we work in terms of the generating series $s_{i,j}(\omega(u))$ from
(\ref{omegadef}), and we use the natural extension of $\pr$ to a homopromphism
$\pr : U(\mf{g})[u] \to U(\mf{p})[u]$.
As the calculations are lengthy, we break them up into a series of lemmas.
Throughout the section we will set
$$
\tilde \imath = 
\widehat{-\imath} = 
\begin{cases}
0&\text{if $i \leq 0$;}\\
1&\text{if $i > 0$.}
\end{cases}
$$

\begin{Lemma} \label{pos_commuter}
Let $y_1, \dots, y_m \in \mf{g}_l$. 
Let $i,j,h,k \in \mathcal{I}_n$.  Let $p,q \in \mathcal{I}_l$.
Then
\begin{align*}
  [s_{i,j}(e_{p,q}),&s_{h,k}(y_1 \otimes \dots \otimes y_m)] \\
  &= \sum_{t=1}^m s_{h,j}(y_1 \otimes \dots \otimes y_{t-1})
       s_{i,k}(e_{p,q} y_t \otimes y_{t+1} \otimes \dots \otimes y_m) \\
  &\quad  - 
       \sum_{t=1}^m s_{h,j}(y_1 \otimes \dots \otimes y_{t-1} 
           \otimes y_t e_{p,q})
       s_{i,k}(y_{t+1} \otimes \dots \otimes y_m) \\
  &\quad 
    + \gamma 
   \left(
     - \sum_{t=1}^m s_{h,-i}(y_1 \otimes \dots \otimes y_{t-1})
       s_{-j,k}(e_{-q,-p} y_t 
           \otimes y_{t+1} \otimes \dots \otimes y_m) \right. \\
  &\qquad \qquad + \left.
       \sum_{t=1}^m s_{h,-i}(y_1 \otimes \dots 
        \otimes y_{t-1} \otimes y_t e_{-q,-p})
       s_{-j,k}(y_{t+1} \otimes \dots \otimes y_m) \right)
\end{align*}
where
\begin{equation} \label{gamma}
  \gamma = \begin{cases}
     \phi^{\hat \imath \hat p + \tilde \imath \tilde{p} +
            \hat \jmath \hat q + \tilde \jmath \tilde{q}}
     \eps^{\hat p + \hat q} & \text{if $p,q \ne 0$;} \\
     \phi^{\hat \jmath \hat q + \tilde \jmath \tilde{q}}
     \eps^{\hat \imath + \hat q} & \text{if $p = 0$, $q \ne 0$;} \\
     \phi^{\hat \imath \hat p + \tilde \imath \tilde{p}}
      \eps^{\hat p + \hat \jmath} & \text{if $p \ne 0$, $q = 0$;} \\
      \eps^{\hat \imath + \hat \jmath} & \text{if $p,q = 0$,} 
   \end{cases}
\end{equation}
and $e_{p,q} y_t, y_t e_{p,q}, e_{-q,-p} y_t$, and $y_t e_{-q,-p}$
denote matrix multiplication in $M_l$.
\end{Lemma}
\begin{proof}
  
First note that for $a,b,c,d \in \mathcal{I}_{n l}$,
\begin{align*}
[f_{a,b}, f_{c,d}] 
&= [e_{a,b} - \eps^{\hat a + \hat b} e_{-b, -a}, e_{c,d} - \eps^{\hat c + \hat d} e_{-d, -c}] \\
&= \delta_{c,b} e_{a,d} - \delta_{b,-d} \eps^{\hat c + \hat d} e_{a,-c} -
   \delta_{-a,c} \eps^{\hat a + \hat b} e_{-b,d} + \delta_{a,d} \eps^{\hat b + \hat c} e_{-b,-c} \\
   &\quad
   - \delta_{a,d} e_{c,b} + \delta_{-c, a} \eps^{\hat c + \hat d} e_{-d, b}
   + \delta_{d,-b} \eps^{\hat a + \hat b} e_{c,-a} - \delta_{c, b} \eps^{\hat a + \hat b} e_{-d, -a} \\
&= \delta_{c,b} f_{a,d} - \delta_{a, d} f_{c,b} + \eps^{\hat a + \hat b} (-\delta_{c,-a} f_{-b,d} + \delta_{-b,d} f_{c,-a})
\end{align*}
Thus for $v, w \in \mathcal{I}_l$ and $a,b,c,d$ such that 
$
\row(a) = i, \col(a) = p,
\row(b) = j, \col(b) = q,
\row(c) = h, \col(c) = v,
\row(d) = k, \col(d) = w
$
we have that
\begin{align*}
  [s_{i,j}&(e_{p,q}), s_{h,k}(e_{v,w})] \\
  &=
[\phi^{\hat \imath \hat p + \hat \jmath \hat q} f_{a,b}, \phi^{\hat h \hat v + \hat k \hat w}f_{c,d}] \\
&=\phi^{\hat \imath \hat p + \hat \jmath \hat q+\hat h \hat v + \hat k \hat w}
 (  \delta_{c,b} f_{a,d} - \delta_{a, d} f_{c,b} + \eps^{\hat a + \hat b} (-\delta_{c,-a} f_{-b,d} + \delta_{-b,d} f_{c,-a})) \\
&=\phi^{\hat \imath \hat p + \hat \jmath \hat q+\hat h \hat v + \hat k \hat w}
 ( \phi^{\hat \imath \hat p + \hat k \hat w} s_{h,j}(1) s_{i,k}(e_{p,q} e_{v,w}) - 
    \phi^{\hat h \hat v + \hat \jmath \hat q} s_{h,j}(e_{v,w} e_{p,q}) s_{i,k}(1) \\
    &\quad + \eps^{\hat a + \hat b}(\phi^{\tilde \jmath \tilde q + \hat k \hat w} s_{h,-i}(1)s_{-j,k}(e_{-q,-p}e_{v,w})
    + \phi^{\hat h \hat v + \tilde \imath \tilde p} s_{h,-i}(e_{v,w} e_{-q,-p}) s_{-j,k}(1))) \\
   &=  s_{h,j}(1) s_{i,k}(e_{p,q} e_{v,w}) -  s_{h,j}(e_{v,w} e_{p,q}) s_{i,k}(1)  \\
   &\quad + \gamma (-s_{h,-i}(1)s_{-j,k}(e_{-q,-p}e_{v,w}) + s_{h,-i}(e_{v,w} e_{-q,-p}) s_{-j,k}(1) ),
\end{align*}
on noting that the $\eps$ term in $\gamma$ equals $\eps^{\hat a + \hat b}$ due to the labeling 
convention specified in the introduction.
Now the linearity 
of $s$ implies the lemma holds for $m=1$ and any $y_1 \in \mf{g}_l$,
and the lemma follows from induction on $m$.
\end{proof}

For $p,q \in \mathcal{I}_l$, let $\Omega_{p,q}(u)$ and $\bar\Omega_{p,q}(u)$ 
denote the square submatrices of $\Omega(u)$ and $\bar\Omega(u)$, respectively,
with rows and columns indexed by $\{p, p+2,\dots,q\}$.

\begin{Lemma} \label{calc_1}
    For each $i,j \in \mathcal{I}_n$ and for $q \in \mathcal{I}_l$ such that 
$q \ge 0$,
\begin{align*}
    \pr & \left( s_{i,j}   
     \left( \rdet
    \begin{pmatrix}
       e_{q+2,q} & e_{q+2,q+2} & e_{q+2,q+4} & \dots & e_{q+2,l-1} \\
       1 & u_{q+2} & e_{q+2,q+4} & \dots & e_{q+2,l-1} \\
       0 & 1 & u_{q+4} & \dots & e_{q+4,l-1} \\
       \vdots & \vdots & \vdots & \ddots & \vdots \\
      0 & 0 & 0 & \dots &  u_{l-1}
    \end{pmatrix}
    \right) \right) \\
    &= (u+\rho_{q+2}-n) s_{i,j}(\rdet \Omega_{q+4,l-1}(u)) \\
    &= (u+\rho_{q+2}-n) s_{i,j}(\rdet \bar\Omega_{q+4,l-1}(u)).
\end{align*}
\end{Lemma}
\begin{proof}
By \eqref{prchi} for any 
$f,g \in \mathcal{I}_n, \pr(s_{f,g}(e_{q+2,q})) = 
\delta_{f,g} = s_{f,g}(1)$.
So
\begin{align}
    \notag
    \pr & \left( s_{i,j}
     \left( \rdet
    \begin{pmatrix}
       e_{q+2,q} & e_{q+2,q+2} & e_{q+2,q+4} & \dots & e_{q+2,l-1} \\
       1 & u_{q+2} & e_{q+2,q+4} & \dots & e_{q+2,l-1} \\
       0 & 1 & u_{q+4} & \dots & e_{q+4,l-1} \\
       \vdots & \vdots & \vdots & \ddots & \vdots \\
      0 & 0 & 0 & \dots &  u_{l-1}
    \end{pmatrix}
    \right) \right) \\ \notag
    &= s_{i,j}
     \left( \rdet
    \begin{pmatrix}
       1 & e_{q+2,q+2} & e_{q+2,q+4} & \dots & e_{q+2,l-1} \\
       1 & u_{q+2} & e_{q+2,q+4} & \dots & e_{q+2,l-1} \\
       0 & 1 & u_{q+4} & \dots & e_{q+4,l-1} \\
       \vdots & \vdots & \vdots & \ddots & \vdots \\
      0 & 0 & 0 & \dots &  u_{l-1}
    \end{pmatrix}
    \right) \\
    \label{l_3} & \quad +\sum_{m \in \mathcal{I}_n} \pr([s_{i,m}(e_{q+2,q}), 
       s_{m,j}(\rdet \Omega_{q+2,l-1}(u))]).
\end{align}
Since $u_{q+2} =  e_{q+2,q+2} + u + \rho_{q+2}$,
doing the obvious row operation gives that
\begin{align}
    \notag
    \rdet
    & \begin{pmatrix}
       1 & e_{q+2,q+2} & e_{q+2,q+4} & \dots & e_{q+2,l-1} \\
       1 & u_{q+2} & e_{q+2,q+4} & \dots & e_{q+2,l-1} \\
       0 & 1 & u_{q+4} & \dots & e_{q+4,l-1} \\
       \vdots & \vdots & \vdots & \ddots & \vdots \\
      0 & 0 & 0 & \dots &  u_{l-1}
    \end{pmatrix} \\
   \notag &= 
    \rdet
    \begin{pmatrix}
       0 & -(u+\rho_{q+2}) & 0 & \dots & 0 \\
       1 & u_{q+2} & e_{q+2,q+4} & \dots & e_{q+2,l-1} \\
       0 & 1 & u_{q+4} & \dots & e_{q+4,l-1} \\
       \vdots & \vdots & \vdots & \ddots & \vdots \\
      0 & 0 & 0 & \dots &  u_{l-1}
    \end{pmatrix} \\
    \label{l_2} &= (u+\rho_{q+2}) \rdet \Omega_{q+4,l-1}(u))
\end{align}
Next we apply Lemma \eqref{pos_commuter} to get that
\[
   [s_{i,m}(e_{q+2,q}), s_{m,j}(\rdet \Omega_{q+2,l-1}(u))]
   = -s_{m,m}(e_{q+2,q}) s_{i,j}(\rdet \Omega_{q+4,l-1}(u)).
\]
By \eqref{prchi} $\pr(s_{m,m}(e_{q+2,q})) = 1$,
so
\begin{equation} \label{l_1}
   \pr([s_{i,m}(e_{q+2,q}), s_{m,j}(\rdet \Omega_{q+2,l-1}(u))])
   = -s_{i,j}(\rdet \Omega_{q+4,l-1}(u)).
\end{equation}
Combining \eqref{l_2} and \eqref{l_1} into \eqref{l_3} gives that
\begin{align*}
    \pr & \left( s_{i,j}   
     \left( \rdet
    \begin{pmatrix}
       e_{q+2,q} & e_{q+2,q+2} & e_{q+2,q+4} & \dots & e_{q+2,l-1} \\
       1 & u_{q+2} & e_{q+2,q+4} & \dots & e_{q+2,l-1} \\
       0 & 1 & u_{q+4} & \dots & e_{q+4,l-1} \\
       \vdots & \vdots & \vdots & \ddots & \vdots \\
      0 & 0 & 0 & \dots &  u_{l-1}
    \end{pmatrix}
    \right) \right) \\
    &= (u+\rho_{q+2}) s_{i,j}(\rdet \Omega_{q+4,l-1}(u))
     - n  s_{i,j}(\rdet \Omega_{q+4,l-1}(u)) \\
    &= (u+\rho_{q+2}-n) s_{i,j}(\rdet \Omega_{q+4,l-1}(u)) \\
    &= (u+\rho_{q+2}-n) s_{i,j}(\rdet \bar\Omega_{q+4,l-1}(u))
\end{align*}
since 
$
  \Omega_{q+4,l-1}(u) =  \bar\Omega_{q+4,l-1}(u)
$
because $q \ge 0$ by assumption.
\end{proof}

\begin{Lemma} \label{calc_2}
For each $i,j,h,k \in \mathcal{I}_n$, for $q \in \mathcal{I}_l$ 
such that $q>0$,
and for $p \in \mathcal{I}_l$ such that $-q < p < q$,
\[
  \pr([s_{i,j}(e_{q+2,q}),s_{h,k}(\rdet\Omega_{p, l-1}(u))]) = 0,
\]
and
\[
  \pr([s_{i,j}(e_{q+2,q}),s_{h,k}(\rdet\bar\Omega_{p, l-1}(u))]) = 0,
\]
\end{Lemma}
\begin{proof}

We shall prove the result for $\Omega(u)$, but note that an identical proof
holds for $\bar\Omega(u)$.  
We compute using Lemma \ref{pos_commuter}
to get that
\[
[s_{i,j}(e_{q+2,q}),s_{h,k}(\rdet\Omega_{p,l-1}(u))] = A - B,
\]
where
\begin{align*}
  A &= s_{h,j}(\rdet \Omega_{p, q-2}(u))
    s_{i,k}
    \left( \rdet
    \begin{pmatrix}
       e_{q+2,q} & e_{q+2,q+2} & \dots & e_{q+2,l-1} \\
       1 & u_{q+2} & \dots & e_{q+2,l-1} \\
       \vdots & \vdots & \ddots & \vdots \\
      0 & 0 & \dots &  u_{l-1}
    \end{pmatrix}
    \right), 
\end{align*}
and
\begin{align*}
  B &= s_{h,j}
    \left( \rdet
    \begin{pmatrix}
       u_{p} & \dots & e_{p, q} & e_{p,q} \\
       \vdots & \ddots & \vdots & \vdots \\
       0 & \dots & u_q & e_{q,q}\\
       0 & \dots & 1 & e_{q+2,q}
    \end{pmatrix}
    \right)
    s_{i,k}(\rdet \Omega_{q+4,l-1}(u)).
\end{align*}
We will be more explicit how to calculate $B$, $A$
is computed in a similar manner.
Let $\mathcal{S}(X)$ denote the symmetric group on a set $X$.
Let $M = \Omega_{p,l-1}(u)$.
By definition,
\[
 \rdet M = \sum_{\sigma \in \mathcal{S}(\{p, p+2, \dots, l-1\})} 
    M_{p, \sigma(p)}
    M_{p+2, \sigma(p+2)} \dots
    M_{l-1, \sigma(l-1)}.
\]
All of the monomials in $B$ come from the second sum in Lemma \ref{pos_commuter} (all the monomials A come from the
first sum in Lemma \ref{pos_commuter}, and the last two sums from that Lemma 
in the calculation of $[s_{i,j}(e_{q+2,q}),s_{h,k}(\rdet\Omega_{p,l-1}(u))]$ are zero).
Furthermore every term in this sum is zero except for those coming from $s_{h,k}$
applied to monomials in $\rdet M$ which contain $e_{v,q+2}$
for some $v \in \mathcal{I}_l, v \le q+2$.
Now since the only nonzero terms of $M$ below the diagonal are scalars occurring
immediately below the diagonal, if $\sigma \in \mathcal{S}(\{p, p+2, \dots, l-1\})$ contributes a nonzero term to 
the second sum of Lemma \ref{pos_commuter},
then 
$
\sigma \in \mathcal{S}(\{p, p+l, \dots, q+2\}) \times \mathcal{S}(\{q+4, q+6, \dots, l-1\}).
$
Thus the sum of the terms of $[s_{i,j}(e_{q+2,q}),s_{h,k}(\rdet\Omega_{p,l-1}(u))]$ which come from the second
sum in Lemma \ref{pos_commuter} is precisely $B$.

Since $\rho_{q+2} - n = \rho_q$, by Lemma \ref{calc_1}, 
\[
  \pr(A) = (u+\rho_q) s_{h,j}(\rdet \Omega_{p, q-2}(u)) 
     s_{i,k}(\rdet \Omega_{q+4,l-1}(u)).
\]
By \eqref{prchi} for any 
$f,g \in \mathcal{I}_n, \pr_{\chi}(s_{f,g}(e_{q+2,q})) = \delta_{f,g} = s_{f,g}(1).$
So the obvious column operation gives that
\begin{align*}
  \pr(B) &= 
   s_{h,j}
    \left( \rdet
    \begin{pmatrix}
       u_{p} & \dots & e_{p, q} & 0 \\
       \vdots & \ddots & \vdots & \vdots \\
       0 & \dots & u_q & -(u+\rho_q) \\
       0 & \dots & 1 & 0
    \end{pmatrix}
    \right) 
    s_{i,k}(\rdet \Omega_{q+4,l-1}(u)) \\
     &=
(u+\rho_q) s_{h,j}(\rdet \Omega_{p, q-2}(u)) 
     s_{i,k}(\rdet \Omega_{q+4,l-1}(u)).
\end{align*}
The lemma now follows.
\end{proof}

\begin{Lemma} \label{allcase}
For each $i,j,h,k \in \mathcal{I}_n$ and for 
$q \in \mathcal{I}_l$ so that $q > 0$,
\[
\pr([s_{i,j}(e_{q+2,q}),s_{h,k}(\rdet\Omega(u))]) = 0
\]
 and
\[
\pr([s_{i,j}(e_{q+2,q}),s_{h,k}(\rdet\bar\Omega(u))]) = 0.
\]
\end{Lemma}
\begin{proof}
We shall prove the result for $\Omega(u)$, but note that an identical proof
holds for $\bar\Omega(u)$.  
We compute using Lemma \ref{pos_commuter}
to get that
\[
[s_{i,j}(e_{q+2,q}),s_{h,k}(\rdet\Omega(u))] = 
A - B + \phi^{\tilde \imath+\tilde \jmath}(-C +D),
\] 
where 
\begin{align*}
  A &= s_{h,j}(\rdet \Omega_{1-l, q-2}(u))
    s_{i,k} 
    \left( \rdet
    \begin{pmatrix}
       e_{q+2,q} & e_{q+2,q+2} & \dots & e_{q+2,l-1} \\
       1 & u_{q+2} & \dots & e_{q+2,l-1} \\
       \vdots & \vdots & \ddots & \vdots \\
      0 & 0 & \dots &  u_{l-1}
    \end{pmatrix}
    \right), \\
  B &= s_{h,j}
    \left( \rdet
    \begin{pmatrix}
       u_{1-l} & \dots & e_{1-l, q} & e_{1-l,q} \\
       \vdots & \ddots & \vdots & \vdots \\
       0 & \dots & u_{q} & e_{q,q}\\
       0 & \dots & 1 & e_{q+2,q}
    \end{pmatrix}
    \right) 
    s_{i,k}(\rdet \Omega_{q+4,l-1}(u)), \\
  C &= s_{h,-i}(\rdet \Omega_{1-l,-q-4}(u))  
    s_{-j,k} 
    \left( \rdet
    \begin{pmatrix}
       e_{-q,-q-2} & e_{-q,-q} & 
         \dots & e_{-q,l-1} \\
       -1 & u_{-q} & \dots & e_{-q,l-1} \\
       \vdots & \vdots & \ddots & \vdots \\
      0 & 0 & \dots &  u_{l-1}
    \end{pmatrix}
    \right),
\end{align*}
and
\begin{align*}
  D &= s_{h,-i} \left( \rdet
    \begin{pmatrix}
       u_{1-l} & \dots & e_{1-l, -q-2} & 
          e_{1-l,-q-2} \\
       \vdots & \ddots & \vdots & \vdots \\
       0 & \dots & u_{-q-2} 
           & e_{-q-2,-q-2} \\
       0 & \dots & -1 & e_{-q,-q-2}
    \end{pmatrix}
    \right)  
    s_{-j,k}(\rdet \Omega_{-q+2,l-1}(u)).
\end{align*}
By Lemma \ref{calc_1},
\[
  \pr(A) = (u+\rho_{q}) s_{h,j}(\rdet \Omega_{1-l,q-2}(u))
      s_{i,k}(\rdet  \Omega_{q+4,l-1}(u)).
\]
The obvious column operation gives that
\begin{align*}
  \pr(B)  
   &= s_{h,j}
    \left( \rdet
    \begin{pmatrix}
       u_{1-l} & \dots & e_{1-l, q} & 0 \\
       \vdots & & \vdots & \vdots \\
       0 & \dots & u_{q} & -(u+\rho_q) \\
       0 & \dots & 1 & 0
    \end{pmatrix}
    \right)  
    s_{i,k}(\rdet \Omega_{q+4,l-1}(u)) \\
    &=
(u+\rho_{q}) s_{h,j}(\rdet \Omega_{1-l,q-2}(u))
      s_{i,k}(\rdet  \Omega_{q+4,l-1}(u)).
\end{align*}
Hence $\pr (A-B) = 0$.

Since by \eqref{prchi} 
$\pr(s_{f,g}(e_{-q,-q-2})) = -\delta_{f,g}
= s_{f,g}(-1)$ for any $f,g \in \mathcal{I}_n$,
we have that
\begin{align} 
  \notag \pr(C) &=  
   s_{h,-i}(\rdet \Omega_{1-l,-q-4}(u))  
   s_{-j,k} 
    \left( \rdet
    \begin{pmatrix}
       -1 & e_{-q,-q} & \dots & e_{-q,l-1} \\
       -1 & u_{-q} & \dots & e_{-q,l-1} \\
       \vdots & \vdots & \ddots & \vdots \\
      0 & 0 & \dots &  u_{l-1}
    \end{pmatrix}
    \right) \\
  \label{l_6} 
     &\quad + 
      \sum_{m \in \mathcal{I}_n} 
   s_{h,-i}(\rdet \Omega_{1-l,-q-4}(u)) 
   \pr([s_{-j,m}(e_{-q,-q-2}),
   s_{m,k}(\rdet \Omega_{-q,l-1}(u))]).
\end{align}
The obvious row operation gives that
\begin{align}  \notag
 s_{-j,k} & \left( \rdet
    \begin{pmatrix}
       -1 & e_{-q,-q} & \dots & e_{-q,l-1} \\
       -1 & u_{-q} & \dots & e_{-q,l-1} \\
       \vdots & \vdots & \ddots & \vdots \\
      0 & 0 & \dots &  u_{l-1}
    \end{pmatrix}
    \right) \\
  \notag
  &= 
 s_{-j,k} \left( \rdet
    \begin{pmatrix}
       0 & -(u + \rho_{-q})  & \dots & 0 \\
       -1 & u_{-q} & \dots & e_{-q,l-1} \\
       \vdots & \vdots & \ddots & \vdots \\
      0 & 0 & \dots &  u_{l-1}
    \end{pmatrix}
    \right) \\
\label{l_4}
  &= 
  -(u+\rho_{-q}) s_{-j,k}(\rdet \Omega_{-q+2,l-1}(u)).
\end{align}
Next we compute using Lemma
\ref{pos_commuter}
to get that
\begin{align*}
  [s_{-j,m}&(e_{-q,-q-2}),
   s_{m,k}(\rdet \Omega_{-q,l-1}(u))] \\
  &= 
   - s_{m,m}(e_{-q,-q-2}) s_{-j,k}(\rdet \Omega_{-q+2,l-1}(u))  
  - A' + B',
\end{align*}
where
\begin{align*}
  A' &=  \phi^{\tilde \jmath+\hat m} s_{m,j}(\rdet \Omega_{-q, q-2}(u)) 
        s_{-m,k}
    \left( \rdet
    \begin{pmatrix}
       e_{q+2,q} &  e_{q+2,q+2} & 
          \dots &  e_{q+2,l-1} \\
       1 & u_{q+2} & \dots & e_{q+2,l-1} \\
       \vdots & \vdots & \ddots & \vdots \\
      0 & 0 & \dots &  u_{l-1}
    \end{pmatrix}
    \right)
\end{align*}
and
\[
  B' = 
    \phi^{\tilde \jmath+\hat m}
    s_{m,j}
    \left( \rdet
    \begin{pmatrix}
       u_{-q} & \dots & e_{-q, q} &  e_{-q,q} \\
       \vdots & \ddots & \vdots & \vdots \\
       0 & \dots & u_{q} &  e_{q,q}\\
       0 & \dots & 1 & e_{q+2,q}
    \end{pmatrix}
    \right) 
    s_{-m,k}(\rdet \Omega_{q+4,l-1}(u)).
\]
By Lemma \ref{calc_1} 
\[
\pr(A') =  \phi^{\tilde \jmath+\hat m} (u+\rho_q) 
  s_{m,j}(\rdet \Omega_{-q, q-2}(u)) 
  s_{-m,k}(\rdet \Omega_{q+4,l-1}(u)).
\]
The usual column operation gives that
\begin{align*}
  \pr(B') &= 
      \phi^{\tilde \jmath+\hat m}
    s_{m,j}
    \left( \rdet
    \begin{pmatrix}
       u_{-q} & \dots & e_{-q, q} &  0 \\
       \vdots & \ddots & \vdots & \vdots \\
       0 & \dots & u_{q} &  -(u+\rho_q)\\
       0 & \dots & 1 & 0
    \end{pmatrix}
    \right) 
    s_{-m,k}(\rdet \Omega_{q+4,l-1}(u)) \\
  &=  \phi^{\tilde \jmath+\hat m} (u + \rho_q) s_{m,j}(\rdet \Omega_{-q, q-2}(u))
       s_{-m,k}(\rdet \Omega_{q+4,l-1}(u)).
\end{align*}
Thus $\pr(-A' + B') = 0$.

By Lemma \ref{calc_2}, we have that
$[s_{m,m}(e_{-q,-q-2}), s_{-j,k}(\rdet \Omega_{-q+2,l-1}(u))] = 0$.
Now since 
$\pr(s_{m,m}(e_{-q,-q-2})) = -1$,
we get that
\[
\pr(s_{m,m}(e_{-q,-q-2}) s_{-j,k}(\rdet \Omega_{-q+2,l-1}(u)))
  = - s_{-j,k}(\rdet \Omega_{-q+2,l-1}(u)).
\]
So
\begin{equation} \label{l_5}
  \pr(  [s_{-j,m}(e_{-q,-q-2})
   s_{m,k}(\rdet \Omega_{-q,l-1}(u))] 
  =  s_{-j,k}(\rdet \Omega_{-q+2,l-1}(u)).
\end{equation}
By combining \eqref{l_4} and \eqref{l_5} into \eqref{l_6}
we get that
\begin{align*}
  \pr(C) &= -(u+\rho_{-q})
  s_{h,-i}(\rdet \Omega_{1-l,-q-4}(u))  
  s_{-j,k}(\rdet  \Omega_{-q+2,l-1}(u)) \\
  &\quad + n
  s_{h,-i}(\rdet \Omega_{1-l,-q-4}(u))  
  s_{-j,k}(\rdet  \Omega_{-q+2,l-1}(u)) \\
  &=
   - (u+\rho_{-q-2})
  s_{h,-i}(\rdet \Omega_{1-l,-q-4}(u))  
  s_{-j,k}(\rdet  \Omega_{-q+2,l-1}(u)).
\end{align*}

Finally, we need to apply $\pr$ to $D$. 
By \eqref{prchi} $\pr(s_{f,g}(e_{-q,-q-2})) = -\delta_{f,g} = s_{f,g}(-1)$ 
for any $f,g \in \mathcal{I}_n$.
By Lemma \ref{calc_2}
$s_{m,-i}(e_{-q,-q-2})$ commutes with 
$s_{-j,k}(\rdet \Omega_{-q+2,l-1}(u))$.
So the usual column operation gives that
\begin{align*}
\pr(D) &= s_{h,-i}
    \left( \rdet
    \begin{pmatrix}
       u_{1-l} & \dots & e_{1-l, -q-2} 
         & 0 \\
       \vdots & \ddots & \vdots & \vdots \\
       0 & \dots & u_{-q-2} & -(u+\rho_{-q-2}) \\
       0 & \dots & -1 & -1
    \end{pmatrix}
    \right)  \\
    &\quad \quad \times
     s_{-j,k}(\rdet \Omega_{-q+2,l-1}(u)) \\
  &= -(u+\rho_{-q-2}) s_{h,-i}(\rdet \Omega_{1-l, -q-4}(u))
     s_{-j,k}(\rdet \Omega_{-q+2,l-1}(u)).
\end{align*}
Thus $\pr(-C+D) = 0$.
\end{proof}

\begin{Lemma} \label{evencase}
Suppose that $l$ is even.
For each $i,j,h,k \in \mathcal{I}_n$
\[
\pr([s_{i,j}(e_{1,-1}),s_{h,k}(\rdet\Omega(u))]) = 0.
\]
\end{Lemma}
\begin{proof}
Since $l$ is even, $\eps = - \phi$, so in all cases by
\eqref{prchi} we have that for all $f,g \in \mathcal{I}_n$
\begin{equation} \label {l_3.5}
  \pr(s_{f,g}(e_{1,-1})) = \delta_{f,g} = s_{f,g}(1).
\end{equation}

We compute using Lemma \ref{pos_commuter}
to get that
\[
[s_{i,j}(e_{1,-1}),s_{h,k}(\rdet\Omega(u))] = A - B + 
   \phi^{\tilde \imath+\hat \jmath} \eps (-C + D),
\]
where
\begin{align*}
  A = &s_{h,j}(\rdet \Omega_{1-l, -3}(u)) 
    s_{i,k} 
    \left( \rdet
    \begin{pmatrix}
       e_{1,-1} & e_{1,1} & \dots & e_{1,l-1} \\
       1 & u_{1} & \dots & e_{1,l-1} \\
       \vdots & \vdots & \ddots & \vdots \\
      0 & 0 & \dots &  u_{l-1}
    \end{pmatrix}
    \right),
\end{align*}
\begin{align*}
  B = &s_{h,j}
    \left( \rdet
    \begin{pmatrix}
       u_{1-l} & \dots & e_{1-l, -1} & e_{1-l,-1} \\
       \vdots & \ddots & \vdots & \vdots \\
       0 & \dots & u_{-1} & e_{-1,-1} \\
       0 & \dots & 1 & e_{1,-1} \\
    \end{pmatrix}
    \right) 
    s_{i,k}(\rdet \Omega_{3,l-1}(u)),
\end{align*}
\begin{align*}
  C = &s_{h,-i}(\rdet \Omega_{1-l,-3}(u)) 
    s_{-j,k} 
    \left( \rdet
    \begin{pmatrix}
       e_{1,-1} & e_{1,1} & \dots & e_{1,l-1} \\
       1 & u_{1} & \dots & e_{1,l-1} \\
       \vdots & \vdots & \ddots & \vdots \\
      0 & 0 & \dots &  u_{l-1}
    \end{pmatrix}
    \right),
\end{align*}
and
\begin{align*}
  D = &s_{h,-i}
    \left( \rdet
    \begin{pmatrix}
       u_{1-l} & \dots &
             e_{1-l, -1} & e_{1-l,-1} \\
       \vdots & \ddots & \vdots & \vdots \\
       0 & \dots & u_{-1} &  e_{-1,-1} \\
       0 & \dots & 1 & e_{1,-1} 
    \end{pmatrix}
    \right) 
    s_{-j,k}(\rdet \Omega_{3,l-1}(u)).
\end{align*}
Consider A first.
Note that
\begin{align}
\notag
\pr &\left( s_{i,k} 
    \left( \rdet
    \begin{pmatrix}
       e_{1,-1} & e_{1,1} & \dots & e_{1,l-1} \\
       1 & u_{1}  & \dots & e_{1,l-1} \\
       \vdots & \vdots & \ddots & \vdots \\
      0 & 0 & \dots &  u_{l-1}
    \end{pmatrix}
    \right)
    \right)  \\
\notag
   &=  
s_{i,k} 
    \left( \rdet
    \begin{pmatrix}
       1 & e_{1,1} & \dots & e_{1,l-1} \\
       1 & u_{1}  & \dots & e_{1,l-1} \\
       \vdots & \vdots & \ddots & \vdots \\
      0 & 0 & \dots &  u_{l-1} 
    \end{pmatrix}
    \right)   \\
\label{A1}
   &\quad +  \sum_{m \in \mathcal{I}_n}
      \pr([s_{i,m}(e_{1,-1}),
   s_{m,k}(\rdet \Omega_{1,l-1}(u))]).
\end{align}
The obvious row operation gives that
\begin{align} \notag
s_{i,k} 
    \left( \rdet
    \begin{pmatrix}
       1 & e_{1,1} & \dots & e_{1,l-1} \\
       1 & u_{1}  & \dots & e_{1,l-1} \\
       \vdots & \vdots & \ddots & \vdots \\
      0 & 0 & \dots &  u_{l-1} 
    \end{pmatrix}
    \right) 
 &= s_{i,k}    \left( \rdet
    \begin{pmatrix}
       0 & -(u+\rho_1) & 0 & \dots & 0 \\
       1 & u_{1}  & e_{1,3} & \dots & e_{1,l-1} \\
       \vdots & \vdots & \vdots & \ddots & \vdots \\
      0 & 0 & 0& \dots &  u_{l-1}
    \end{pmatrix}
    \right)  \\
   \label{l_3.6} &= 
       (u+\rho_1) s_{i,k}(\rdet \Omega_{3,l-1}(u)).
\end{align}
Next consider the terms 
$\pr([s_{i,m}(e_{1,-1}), s_{m,k}(\rdet \Omega_{1,l-1}(u))])$
from \eqref{A1}.
We calculate using Lemma \ref{pos_commuter} to get that
\begin{align*}
[s_{i,m}&(e_{1,-1}), s_{m,k}(\rdet \Omega_{1,l-1}(u))]\\
 &= -s_{m,m}(e_{1,-1}) s_{i,k}(\rdet \Omega_{3,l-1}(u)) 
  + \phi^{\tilde \imath + \hat m} \eps s_{m,-i}(e_{1,-1}) s_{-m,k}(\rdet \Omega_{3,l-1}(u)).
\end{align*}
So
\begin{align} \notag
  \pr&([s_{i,m}(e_{1,-1}), s_{m,k}(\rdet \Omega_{1,l-1}(u))]) \\
    \label{l_3.7} 
      &= 
      -s_{i,k}(\rdet \Omega_{3,l-1}(u))
          +  \phi^{\tilde \imath + \hat m} \eps \delta_{m,-i}  
              s_{-m,k}(\rdet \Omega_{3,l-1}(u)).
\end{align}
So by combining \eqref{l_3.7} and \eqref{l_3.6} in \eqref{A1} we
we get that
\begin{align}  \notag
  \pr(A) &=  
       (u+\rho_1) 
          s_{h,j}( \rdet \Omega_{1-l, -3}(u))
          s_{i,k}(\rdet \Omega_{3,l-1}(u)) \\
   \notag
   &\quad -n 
          s_{h,j}( \rdet \Omega_{1-l, -3}(u))
   s_{i,k}(\rdet \Omega_{3,l-1}(u))  \\
   \notag
    &\quad + \eps 
          s_{h,j}( \rdet \Omega_{1-l, -3}(u))
   s_{i,k}(\rdet \Omega_{3,l-1}(u))  \\
\label{l_3.8}
   &=
       (u+\rho_{-1}) 
          s_{h,j}( \rdet \Omega_{1-l, -3}(u))
          s_{i,k}(\rdet \Omega_{3,l-1}(u)),
\end{align}
since $\rho_1 -n + \eps = \rho_{-1}$. 

Next we consider $B$.
The usual column operation gives that
\begin{align} \notag
  \pr(B)  
   &= s_{h,j}
    \left( \rdet
    \begin{pmatrix}
       u_{1-l} & \dots & e_{1-l, -1} & e_{1-l,-1} \\
       \vdots & \ddots & \vdots & \vdots \\
       0 & \dots & u_{-1} & e_{-1,-1} \\
       0 & \dots & 1 & 1 \\
    \end{pmatrix}
    \right) 
    s_{i,k}(\rdet \Omega_{3,l-1}(u)) \\
   \notag 
   &= s_{h,j}
    \left( \rdet
    \begin{pmatrix}
       u_{1-l} & \dots & e_{1-l, -1} & 0 \\
       \vdots & \ddots & \vdots & \vdots \\
       0 & \dots & e_{-1, -1} & -(u+\rho_{-1}) \\
       0 & \dots & 1 & 0 \\
    \end{pmatrix}
    \right) 
  s_{i,k}(\rdet \Omega_{3,l-1}(u)) \\
  \label{l_3.9}   &=  
        (u+\rho_{-1}) s_{h,j}(\rdet \Omega_{1-l,-3}(u))
        s_{i,k}(\rdet \Omega_{3,l-1}(u)).
\end{align}
So by \eqref{l_3.8} and \eqref{l_3.9}, $\pr(A-B) = 0$.

Next consider $C$.  Since $C$ is nearly identical to $A$, 
an argument nearly identical to that
used for $A$ shows that
\[
  \pr(C) = (u+\rho_{-1}) s_{h,-i}( \rdet \Omega_{1-l, -3}(u))
       s_{-j,k}(\rdet \Omega_{3,l-1}(u)).
\]
Since $D$ is nearly identical to $B$, an argument nearly identical to that
used for $B$ shows that
\[
  \pr(D) = (u+\rho_{-1}) 
           s_{h,-i}( \rdet \Omega_{1-l, -3}(u))
          s_{-j,k}(\rdet \Omega_{3,l-1}(u)).
\]
So $\pr(-C+D) = 0$.
\end{proof}

\begin{Lemma} \label{oddcase}
Suppose that $l$ is odd.
For $i,j,h,k \in \mathcal{I}_n$,
\begin{align*}
   \pr&([s_{i,j}(e_{2,0}),s_{h,k}(\rdet\Omega(u))])  \\
  &= 
  \phi /2 s_{h,j}(\rdet \Omega_{1-l,-2}(u))
            s_{i,k}(\rdet \Omega_{4,l-1}(u)) \\
   &\quad +\phi^{\tilde \imath + \hat \jmath+1}  /2
      s_{h,-i}(\rdet \Omega_{1-l, -4}(u))
     s_{-j,k}(\rdet \Omega_{4,l-1}(u)) \\
   &\quad - \phi^{\tilde \imath + \hat \jmath}
         /2 s_{h,-i}(\rdet \Omega_{1-l, -4}(u))
     s_{-j,k}(\rdet \Omega_{2,l-1}(u))
\end{align*}
and
\begin{align*}
\pr&([s_{i,j}(e_{2,0}),s_{h,k}(\rdet\bar\Omega(u))])  \\
  &= 
  (u+\phi/2) s_{h,j}(\rdet \Omega_{1-l,-2}(u))
            s_{i,k}( \rdet \Omega_{4,l-1}(u)) \\
   &\quad + \phi^{\tilde \imath + \hat \jmath}
       (u+\phi/2)  s_{h,-i}(\rdet \Omega_{1-l, -4}(u))
     s_{-j,k}(\rdet \Omega_{4,l-1}(u)) \\
   &\quad - \phi^{\tilde \imath + \hat \jmath+1}
     (u+\phi/2)  s_{h,-i}(\rdet \Omega_{1-l, -4}(u))
     s_{-j,k}(\rdet \Omega_{2,l-1}(u)).
\end{align*}
\end{Lemma}
\begin{proof}
Since $l$ is odd, $\eps = \phi$.
We compute using Lemma \ref{pos_commuter}
to get that
\[
[s_{i,j}(e_{2,0}),s_{h,k}(\rdet\Omega(u))] 
  = A - B + \phi^{\tilde \imath + \hat \jmath} (-C + D),
\]
and
\[
[s_{i,j}(e_{2,0}),s_{h,k}(\rdet\bar\Omega(u))] = 
    \bar A - \bar B +  \phi^{\tilde \imath+ \hat \jmath} (- \bar C + \bar D),
\]
where 
\begin{align*}
  A = \bar A = &s_{h,j}(\rdet \Omega_{1-l, -2}(u)) 
    s_{i,k} 
    \left( \rdet
    \begin{pmatrix}
       e_{2,0} & e_{2,2} & \dots & e_{2,l-1} \\
       1 & u_{2} & \dots & e_{2,l-1} \\
       \vdots & \vdots & \ddots & \vdots \\
      0 & 0 & \dots &  u_{l-1}
    \end{pmatrix}
    \right),
\end{align*}
\begin{align*}
  B = &s_{h,j}
    \left( \rdet
    \begin{pmatrix}
       u_{1-l} & \dots & e_{1-l, 0} & e_{1-l,0} \\
       \vdots & \ddots & \vdots & \vdots \\
       0 & \dots & e_{0,0} +u & e_{0,0} \\
       0 & \dots & 1 & e_{2,0}
    \end{pmatrix}
    \right) 
    s_{i,k}(\rdet \Omega_{4,l-1}(u)),
\end{align*}
\begin{align*}
  \bar B = &s_{h,j}
    \left( \rdet
    \begin{pmatrix}
       u_{1-l} & \dots &  e_{1-l, 0} & e_{1-l,0} \\
       \vdots & \ddots & \vdots & \vdots \\
       0 & \dots & e_{0,0} & e_{0,0} \\
       0 & \dots & 1 & e_{2,0}
    \end{pmatrix}
    \right) 
    s_{i,k}(\rdet \Omega_{4,l-1}(u)),
\end{align*}
\begin{align*}
  C &= s_{h,-i}(\rdet \Omega_{1-l,-4}(u))  
    s_{-j,k} 
    \left( \rdet
    \begin{pmatrix}
       e_{0,-2} & e_{0,0} & \dots & e_{0,l-1} \\
       - \phi & e_{0,0} +u & \dots & e_{0,l-1} \\
       \vdots & \vdots & \ddots & \vdots \\
      0 & 0 & \dots &  u_{l-1}
    \end{pmatrix}
    \right),
\end{align*}
\begin{align*}
  \bar C &= s_{h,-i}(\rdet \Omega_{1-l,-4}(u))  
    s_{-j,k} 
    \left( \rdet
    \begin{pmatrix}
        e_{0,-2} & e_{0,0} & \dots &  e_{0,l-1} \\
       -\phi & e_{0,0} & \dots & e_{0,l-1} \\
       \vdots & \vdots & \ddots & \vdots \\
      0 & 0 & \dots &  u_{l-1}
    \end{pmatrix}
    \right),
\end{align*}
and
\begin{align*}
  D = \bar D &= s_{h,-i}
    \left( \rdet
    \begin{pmatrix}
       u_{1-l}& \dots & e_{1-l, -2} & e_{1-l,-2} \\
       \vdots & & \vdots & \vdots \\
       0 & \dots & u_{-2} & e_{-2,-2} \\
       0 & \dots & -\phi &  e_{0,-2}
    \end{pmatrix}
    \right)  
    s_{-j,k}(\rdet \Omega_{2,l-1}(u)).
\end{align*}

By Lemma \ref{calc_1}
\begin{align*}
 \pr(A) &= \pr(\bar A) = 
   (u+ \rho_2-n ) s_{h,j}(\rdet \Omega_{1-l, -2}(u))
   s_{i,k} (\rdet \Omega_{4,l-1}(u)) \\
   &= (u - \phi /2) s_{h,j}(\rdet \Omega_{1-l, -2}(u))
   s_{i,k} (\rdet \Omega_{4,l-1}(u)).
\end{align*}
By \eqref{prchi} for any $f,g \in \mathcal{I}_n$, 
$\pr(s_{f,g}(e_{2,0})) = \delta_{f,g} = s_{f,g}(1)$.
So the obvious column operation gives that
\begin{align*}
  \pr(B) &= s_{h,j}
    \left( \rdet
    \begin{pmatrix}
       u_{1-l} & \dots & e_{1-l, 0} & 0 \\
       \vdots & \ddots & \vdots & \vdots \\
       0 & \dots & e_{0,0} +u & -u \\
       0 & \dots & 1 & 0 
    \end{pmatrix}
    \right) 
    s_{i,k}(\rdet \Omega_{4,l-1}(u)) \\
 &=
u s_{h,j}(\rdet \Omega_{1-l,-2}(u)) 
         s_{i,k}(\rdet \Omega_{4, l-1}(u))
\end{align*}
and
\[
  \pr(\bar B)  = 0.
\]
So 
\begin{equation} \label{l_6.10}
  \pr(A-B) = -\phi /2  s_{h,j}(\rdet \Omega_{1-l, -2}(u))
   s_{i,k} (\rdet \Omega_{4,l-1}(u)),
\end{equation}
and
\begin{equation} \label{l_6.11}
  \pr(\bar A-\bar B) = (u - \phi /2)  s_{h,j}(\rdet \Omega_{1-l, -2}(u))
   s_{i,k} (\rdet \Omega_{4,l-1}(u)).
\end{equation}

Next we consider $\pr(C)$.
Since $\eps = \phi$, in all cases we have by \eqref{prchi} that
for any $f,g \in \mathcal{I}_n$,
$\pr(s_{f,g}(e_{0,-2}))  = -\phi \delta_{f,g} 
= s_{f,g}(-\phi)$. 
So we have that
\begin{align} \notag
  \pr(C) &=  
     s_{h,-i}(\rdet \Omega_{1-l,-4}(u))
    s_{-j,k}
    \left( \rdet
    \begin{pmatrix}
       -\phi  & e_{0,0} & \dots & e_{0,l-1} \\
       -\phi  & e_{0,0} + u & \dots & e_{0,l-1} \\
       \vdots & \vdots & \ddots & \vdots \\
      0 & 0 & \dots &  u_{l-1}
    \end{pmatrix}
    \right) \\
    \label{l_6.12} &\quad + 
         \sum_{m \in \mathcal{I}_n} 
         s_{h,-i}(\rdet \Omega_{1-l,-4}(u))
           \pr([s_{-j,m}(e_{0,-2}), 
                 s_{m,k} (\rdet \Omega_{0,l-1}(u))]).
\end{align}
The obvious row operation gives that
\begin{align} \notag 
  s_{-j,k}
    &\left( \rdet
    \begin{pmatrix}
       -\phi  & e_{0,0} & \dots & e_{0,l-1} \\
       -\phi  & e_{0,0} + u & \dots & e_{0,l-1} \\
       \vdots & \vdots & \ddots & \vdots \\
      0 & 0 & \dots &  u_{l-1}
    \end{pmatrix}
    \right)  \\
  \notag &=
  s_{-j,k}
    \left( \rdet
    \begin{pmatrix}
       0 & -u & \dots & 0  \\
       -\phi  & e_{0,0} + u & \dots & e_{0,l-1} \\
       \vdots & \vdots & \ddots & \vdots \\
      0 & 0 & \dots &  u_{l-1}
    \end{pmatrix}
    \right) \\
    \label{l_6.18} &=  -\phi u s_{-j,k}(\rdet \Omega_{2,l-1}(u)).
\end{align}
Next we consider the terms  
$[s_{-j,m}(e_{0,-2}), s_{m,k} (\rdet \Omega_{0,l-1}(u))]$
from \eqref{l_6.12}.
By applying Lemma \ref{pos_commuter}, we compute that
\begin{align} \notag
[s_{-j,m}&(e_{0,-2}), s_{m,k} (\rdet \Omega_{0,l-1}(u))] \\
\notag &= - s_{m,m}(e_{0,-2}) s_{-j,k} (\rdet \Omega_{2,l-1}(u)) \\
  \notag &\quad - \phi^{\tilde \jmath+1+\hat m} 
       \delta_{m,j} s_{-m,k}
    \left( \rdet
        \begin{pmatrix}
       e_{2,0} & e_{2,2} & \dots & e_{2,l-1} \\
       1 & u_{2} & \dots & e_{2,l-1} \\
       \vdots & \vdots & \ddots & \vdots \\
      0 & 0 & \dots &  u_{l-1}
    \end{pmatrix}
    \right) \\
   \label{l_6.13} &\quad + \phi^{\tilde \jmath+1+ \hat m} s_{m,j}
   \left( \rdet
     \begin{pmatrix}
       e_{0,0} + u & e_{0,0} \\
       1 & e_{2,0}
     \end{pmatrix}
    \right)
    s_{-m,k} (\rdet \Omega_{4,l-1}(u)).
\end{align}
We need to apply $\pr$ to each term of this expression.
First we use Lemma \ref{pos_commuter} again to get that
\begin{align} \notag 
\pr( &s_{m,m}(e_{0,-2}) s_{-j,k} (\rdet \Omega_{2,l-1}(u))) \\
  \notag &= -\phi s_{-j,k} (\rdet \Omega_{2,l-1}(u))
   + \pr([ s_{m,m}(e_{0,-2}), s_{-j,k} (\rdet \Omega_{2,l-1}(u))]) \\
  \notag &= -\phi s_{-j,k} (\rdet \Omega_{2,l-1}(u))
    + \phi^{\hat {m}+1+ \hat m} 
      \pr(s_{-j,-m}(e_{2,0}) s_{-m, k}(\rdet \Omega_{4,l-1}(u))) \\
  \label{l_6.14} &= -\phi s_{-j,k} (\rdet \Omega_{2,l-1}(u))
    + \phi 
     \delta_{j,m}  s_{-m, k}(\rdet \Omega_{4,l-1}(u)).
\end{align}
Next by applying Lemma \ref{calc_1}, we have that
\begin{align} \notag 
 \pr \left( s_{-m,k}
    \left( \rdet
        \begin{pmatrix}
       e_{2,0} & e_{2,2} & \dots & e_{2,l-1} \\
       1 & u_{2} & \dots & e_{2,l-1} \\
       \vdots & \vdots & \ddots & \vdots \\
      0 & 0 & \dots &  u_{l-1}
    \end{pmatrix}
    \right)
    \right)
   &= (u+\rho_2 - n) s_{-m,k}(\rdet \Omega_{4,l-1}(u)) \\
   \label{l_6.15} &= (u-\phi /2) s_{-m,k}(\rdet \Omega_{4,l-1}(u)).
\end{align}
Next note that
\begin{align} \notag
  \pr &\left( s_{m,j}
  \left( \rdet
     \begin{pmatrix}
       e_{0,0} + u & e_{0,0} \\
       1 & e_{2,0}
     \end{pmatrix}
    \right)
    s_{-m,k} (\rdet \Omega_{4,l-1}(u))
    \right) \\
   \label{l_6.16} &=
  u \delta_{m,j} s_{-m,k} (\rdet \Omega_{4,l-1}(u)).
\end{align}
So by combining \eqref{l_6.14}, \eqref{l_6.15}, and \eqref{l_6.16} in
\eqref{l_6.13} we get that
\begin{align} \notag
\pr(&[s_{-j,m}(e_{0,-2}), 
              s_{m,k} (\rdet \Omega_{0,l-1}(u))]) \\
  \notag &= \phi s_{-j,k} (\rdet \Omega_{2,l-1}(u))
     -  \phi
         \delta_{j,m} s_{-m, k}(\rdet \Omega_{4,l-1}(u)) \\
     \notag &\quad - \phi^{\tilde \jmath + 1+\hat m} \delta_{m,j}  
         (u-\phi/2) s_{-m,k}(\rdet \Omega_{4,l-1}(u)) \\
     \label{l_6.17} &\quad +   \phi^{\tilde \jmath+1+\hat m}
         u \delta_{m,j} s_{-m,k} (\rdet \Omega_{4,l-1}(u)).
\end{align}
So by combining \eqref{l_6.18} and \eqref{l_6.17} in \eqref{l_6.12} we get that
\begin{align} \notag
  \pr(C) &= 
    - \phi u 
      s_{h,-i}(\rdet \Omega_{1-l,-4}(u))
      s_{-j,k}(\rdet \Omega_{2,l-1}(u)) \\
    \notag &\quad + \phi n
      s_{h,-i}(\rdet \Omega_{1-l,-4}(u))
      s_{-j,k} (\rdet \Omega_{2,l-1}(u)) \\
    \notag &\quad - \phi
      s_{h,-i}(\rdet \Omega_{1-l,-4}(u))
      s_{-j,k}(\rdet \Omega_{4,l-1}(u)) \\
    \notag &\quad - \phi^{\tilde \jmath + 1+\hat \jmath} (u-\phi/2)
      s_{h,-i}(\rdet \Omega_{1-l,-4}(u))
       s_{-j,k}(\rdet \Omega_{4,l-1}(u)) \\
    \notag &\quad + \phi^{\tilde \jmath+1+\hat \jmath} u
      s_{h,-i}(\rdet \Omega_{1-l,-4}(u))
       s_{-j,k}(\rdet \Omega_{4,l-1}(u)) \\
   \notag &=
   -\phi (u-n)
      s_{h,-i}(\rdet \Omega_{1-l,-4}(u))
      s_{-j,k}(\rdet \Omega_{2,l-1}(u)) \\
   \label{l_6.19} &\quad  -\phi /2  
      s_{h,-i}(\rdet \Omega_{1-l,-4}(u))
      s_{-j,k}(\rdet \Omega_{4,l-1}(u)).
\end{align}
For the last equality we use that $\phi^{\tilde \jmath +\hat \jmath} = \phi$,
since $j$ cannot be zero if $\phi = -1$.

A very similar calculation shows that
\begin{align} \notag
  \pr(\bar C) &= 
     \phi  n
      s_{h,-i}(\rdet \Omega_{1-l,-4}(u))
      s_{-j,k} (\rdet \Omega_{2,l-1}(u)) \\
    \label{l_6.24} &\quad  -(u+\phi/2)
      s_{h,-i}(\rdet \Omega_{1-l,-4}(u))
       s_{-j,k}(\rdet \Omega_{4,l-1}(u)).
\end{align}

Finally we must calculate $\pr(D)$.  Note that
\begin{align} \notag
  \pr(D) &= \pr(\bar D)   \\
    \notag &= s_{h,-i} \left( \rdet
    \begin{pmatrix}
       u_{1-l}& \dots & e_{1-l, -2} & e_{1-l,-2} \\
       \vdots & \ddots & \vdots & \vdots \\
       0 & \dots & u_{-2} &  e_{-2,-2} \\
       0 & \dots & -\phi & -\phi
    \end{pmatrix}
    \right)  
    s_{-j,k}(\rdet \Omega_{2,l-1}(u)) \\
    \label{l_6.21} &\quad +
    \sum_{m \in \mathcal{I}_n} 
    s_{h,m} ( \rdet \Omega_{1-l, -2}(u)) 
       \pr([s_{m,-i}(e_{0,-2}), s_{-j,k}(\rdet \Omega_{2,l-1}(u))]).
\end{align}
The obvious column operation gives that
\begin{align}
    \notag s_{h,-i} &
    \left( \rdet
    \begin{pmatrix}
       u_{1-l}& \dots & e_{1-l, -2} & 
              e_{1-l,-2} \\
       \vdots & \ddots & \vdots & \vdots \\
       0 & \dots & u_{-2} & e_{-2,-2} \\
       0 & \dots & -\phi & -\phi
    \end{pmatrix}
    \right)  \\
   \notag &= 
    s_{h,-i}
    \left( \rdet
    \begin{pmatrix}
       u_{1-l}& \dots & e_{1-l, -2} & 0 \\
       \vdots & \ddots & \vdots & \vdots \\
       0 & \dots & u_{-2} & -(u+\rho_{-2}) \\
       0 & \dots & -\phi & 0
    \end{pmatrix}
    \right)  \\
    \label{l_6.20} &=
       -\phi (u+\rho_{-2})  s_{h,-i}(\rdet \Omega_{1-l, -4}(u)).
\end{align}
Next we consider the terms 
$\pr([s_{m,-i}(e_{0,-2}), s_{-j,k}(\rdet \Omega_{2,l-1}(u))])$
from \eqref{l_6.21}.
We compute using Lemma \ref{pos_commuter} to get that 
\begin{align} \notag
  \pr([s_{m,-i}(e_{0,-2}), s_{-j,k}(\rdet \Omega_{2,l-1}(u))])  
  &=
  \phi^{\hat m+1+\tilde \imath} \pr(s_{-j, -m}(e_{2,0})
    s_{i,k}(\rdet \Omega_{4,l-1}(u))) \\
  \label{l_6.22} &=  \phi^{\hat m+1+\tilde \imath} 
        \delta_{j,m} s_{i,k}(\rdet \Omega_{4,l-1}(u)).
\end{align}
So by combining \eqref{l_6.20} and \eqref{l_6.22} in \eqref{l_6.21}
we have that
\begin{align} \notag
  \pr(D) = \pr(\bar D)
   &= - \phi 
          (u+\rho_{-2})  s_{h,-i}(\rdet \Omega_{1-l, -4}(u))
     s_{-j,k}(\rdet \Omega_{2,l-1}(u)) \\
   \label{l_6.23} &\quad + \phi^{\hat \jmath+1+\tilde \imath}
            s_{h,j}(\rdet \Omega_{1-l,-2}(u)) 
            s_{i,k}((\rdet \Omega_{4,l-1}(u)).
\end{align}

So by \eqref{l_6.10}, \eqref{l_6.19}, and \eqref{l_6.23} we have that
\begin{align*}
\pr(A-B+\phi^{\tilde \imath+\hat \jmath} (-C+ D)) 
   &=
  \phi /2 s_{h,j}(\rdet \Omega_{1-l,-2}(u))
            s_{i,k}(\rdet \Omega_{4,l-1}(u)) \\
   &\quad +\phi^{\tilde \imath + \hat \jmath+1} /2  
      s_{h,-i}(\rdet \Omega_{1-l, -4}(u))
     s_{-j,k}(\rdet \Omega_{4,l-1}(u)) \\
   &\quad - \phi^{\tilde \imath +\hat \jmath} 
         /2 s_{h,-i}(\rdet \Omega_{1-l, -4}(u))
     s_{-j,k}(\rdet \Omega_{2,l-1}(u)).
\end{align*}

By \eqref{l_6.11}, \eqref{l_6.24}, and \eqref{l_6.23} we have that
\begin{align*}
\pr(\bar A-\bar B+&\phi^{\tilde \imath+\hat \jmath}(-\bar C+\bar D))\\ &=
  (u+\phi/2) s_{h,j}(\rdet \Omega_{1-l,-2}(u))
            s_{i,k}(\rdet \Omega_{4,l-1}(u)) \\
   &\quad + \phi^{\tilde \imath + \hat \jmath}
       (u+\phi/2)  s_{h,-i}(\rdet \Omega_{1-l, -4}(u))
     s_{-j,k}(\rdet \Omega_{4,l-1}(u)) \\
   &\quad - \phi^{\tilde \imath+ \hat \jmath+1} 
     (u+\phi/2)  s_{h,-i}(\rdet \Omega_{1-l, -4}(u))
     s_{-j,k}(\rdet \Omega_{2,l-1}(u)).
\end{align*}
\end{proof}

Now we can prove Theorem~\ref{containedinW}.
We need to show that the equation (\ref{toprove}) holds for all elements
$x$ lying in the generating set (\ref{thegens}) for $\mathfrak{m}$.
This follows from Lemmas~\ref{allcase}, \ref{evencase} and \ref{oddcase},
using the definition of $\omega(u)$ from \eqref{omegadef}.


\end{document}